\documentclass[11pt,twoside]{amsart}
\usepackage{mathrsfs}
\usepackage{graphicx}
\usepackage{mathrsfs}
\usepackage{amsmath}
\usepackage{amsthm}
\usepackage{amsfonts}
\usepackage{amssymb}
\usepackage{latexsym}
\usepackage[all]{xy}
\usepackage[colorlinks=true]{hyperref}
\hypersetup{urlcolor=blue, citecolor=red}

\date{}
\allowdisplaybreaks[4] \footskip=15pt
\renewcommand{\uppercasenonmath}[1]{}

\numberwithin{equation}{section} \theoremstyle{plain}
\newtheorem{theorem}{Theorem}[section]
\newtheorem{corollary}[theorem]{Corollary}
\newtheorem{lemma}[theorem]{Lemma}
\newtheorem{proposition}[theorem]{Proposition}
\theoremstyle{definition}
\newtheorem{definition}[theorem]{Definition}

\newtheorem{remark}[theorem]{Remark}
\newtheorem*{ack*}{ACKNOWLEDGEMENTS}

\newcommand{\pf}{\noindent\begin {proof}}
\newcommand{\epf}{\end{proof}}

\begin{document}
\begin{center}
{
{\bf\large Recollements associated to cotorsion pairs over upper triangular matrix rings}\\

\vspace{0.5cm}    Rongmin Zhu$^{1}$, Yeyang Peng$^{1}$, Nanqing Ding$^{1}$ \\

$^{1}$Department of Mathematics, Nanjing University, Nanjing 210093, China
}
\end{center}
\title{}\maketitle\footnote[0]{Corresponding author: Rongmin Zhu.

E-mail addresses: rongminzhu@hotmail.com, pengyy@smail.nju.edu.cn, nqding@nju.edu.cn. }
\vspace{-3em}
 $$\bf  Abstract$$
\leftskip10truemm \rightskip10truemm \noindent Let $A$, $B$ be two rings and $T=\left(\begin{smallmatrix}  A & M \\  0 & B \\\end{smallmatrix}\right)$ with $M$ an $A$-$B$-bimodule. Given two complete hereditary cotorsion pairs $(\mathcal{A}_{A},\mathcal{B}_{A})$ and $(\mathcal{C}_{B},\mathcal{D}_{B})$ in $A$-Mod and $B$-Mod respectively. We define two cotorsion pairs $(\Phi(\mathcal{A}_{A},\mathcal{C}_{B}), \mathrm{Rep}(\mathcal{B}_{A},\mathcal{D}_{B}))$ and $(\mathrm{Rep}(\mathcal{A}_{A},\mathcal{C}_{B}),$\\$\Psi(\mathcal{B}_{A},\mathcal{D}_{B}))$ in $T$-Mod and show that both of these cotorsion pairs are complete and hereditary. Given two cofibrantly generated model structures $\mathcal{M}_{A}$ and $\mathcal{M}_{B}$ on $A$-Mod and $B$-Mod respectively. Using the result above, we investigate when there exist a cofibrantly generated model structure $\mathcal{M}_{T}$ on $T$-Mod and a recollement of $\mathrm{Ho}(\mathcal{M}_{T})$ relative to $\mathrm{Ho}(\mathcal{M}_{A})$ and $\mathrm{Ho}(\mathcal{M}_{B})$. Finally, some applications are given in Gorenstein homological algebra.  \\
\vbox to 0.3cm{}

\leftskip0truemm \rightskip0truemm
\bigskip

\section { \bf Introduction }
\leftskip0truemm \rightskip0truemm
 Let $A$ and $B$ be two rings. For any bimodule $_{A}M_{B}$, we write $T$ for
the upper triangular matrix ring $\left(\begin{smallmatrix}  A & M \\  0 & B \\\end{smallmatrix}\right)$. Such rings play an important role in the
study of the representation theory of artin rings and algebras. It is known that each $T$-module is identified with a
triple $\left(\begin{smallmatrix}  X \\  Y  \\\end{smallmatrix}\right)_{\phi}$, where $X \in A$-$\mathrm{Mod}$, $Y \in B$-$\mathrm{Mod}$ and $ \phi: M\otimes_{B} Y\rightarrow X$ is a homomorphism of $A$-modules (see  \cite[Theorem 1.5]{gr82}). Denote by $\psi$ the corresponding homomorphism from $Y$ to $\mathrm{Hom}_{A}(M,X)$ by adjoint isomorphism. Some important
classes of modules over upper triangular matrix rings have been studied by many authors (e.g., see \cite{ah00}, \cite{ah99}, \cite{xi12}, \cite{zh13} and \cite{ee14} and their references).

Now we recall the characterizations of the following classes of left $T$-modules.

Let $X=\binom{X_{1}}{X_{2}}_{\phi}$ be a left $T$-module.

(1) (\cite[Theorem 3.1]{ah00}) $X$ is projective if and only if $X_{2}$ is projective in $B$-Mod, $\mathrm{coker}\phi$ is projective in $A$-Mod and $\phi$ is monomorphic.

(2) (\cite[Proposition 5.1]{ah99}) $X$ is injective if and only if $X_{1}$ is injective in $A$-Mod, $\mathrm{ker}\psi$ is injective in $B$-Mod and $\psi$ is epimorphic.

(3) (\cite{fo75}, \cite[Theorem 2.5]{ee11}) $X$ is flat if and only if $X_{2}$ is flat in $B$-Mod, $\mathrm{coker}\phi$ is flat in $A$-Mod and $\phi$ is monomorphic.

(4) (\cite[Theorem 3.5]{ee14}) Suppose that $_{A}M$ has finite projective dimension, $M_{B}$  has finite flat dimension and  $A$ is left Gorenstein regular (see  \cite[Definition 2.1]{ee14} or Section 4 below). Then $X$ is Gorenstein projective if and only if $X_{2}$ and $\mathrm{coker}\phi$ are Gorenstein projective and the homomorphism $\phi$ is monomorphic.

(5) (\cite[Theorem 3.8]{ee14}) Suppose that $_{A}M$ has finite projective dimension, $M_{B}$  has finite flat dimension and  $B$ is left Gorenstein regular. Then $X$ is Gorenstein injective if and only if $X_{1}$ and $\mathrm{ker}\psi$ are Gorenstein injective and the homomorphism $\psi$ is epimorphic.

Let $R$ be a ring. Recall that a cotorsion pair in $R$-Mod is a pair $(\mathcal{A},\mathcal{B})$ of classes of $R$-modules which are
orthogonal with respect to Ext$^{1}_{R}(-,-)$.
Denote by $\mathcal{P}_{R}$, $\mathcal{I}_{R}$,  $\mathcal{GP}_{R}$, $\mathcal{GI}_{R}$ the classes of
projective, injective,  Gorenstein projective and Gorenstein injective left $R$-modules, respectively. It is known that $(\mathcal{P}_{R},R\text{-Mod})$ and $(R\text{-Mod},\mathcal{I}_{R})$ are complete hereditary cotorsion pairs. If $R$ is a ring with all projective left $R$-modules having finite injective dimension, then $(\mathcal{GP}_{R},\mathcal{GP}_{R}^{\perp})$ is a complete hereditary cotorsion pair in $R$-Mod by \cite[Theorem 4.2]{wa16}. It is shown in \cite[Theorem 4.6]{js18} that $(^{\perp}\mathcal{GI}_{R},\mathcal{GI}_{R})$ is a complete cotorsion pair for any ring. By the characterizations above, it seems that the class of $T$-modules as the left (resp. right) half of a cotorsion pair  shares the same descriptions under some conditions.

On the other hand, the notion of a torsion pair in an abelian category was introduced by Dickson \cite{di66}. It plays a prominent role in representation theory of algebras. Campbell \cite{ca78} proved that
the torsion pairs for $T$-mod correspond bijectively to pairs of torsion
pairs, one for $A$-mod and one for $B$-mod. This motivates us to investigate the relationship of cotorsion pairs among $A$-Mod, $B$-Mod and $T$-Mod.
Given a class $\mathcal{C}$ of modules, we write $\mathcal{C}_{A}$ (resp. $\mathcal{C}_{B}$) instead of $\mathcal{C}$ if $\mathcal{C}\subseteq A$-Mod (resp. $\mathcal{C}\subseteq B$-Mod).
Let $\mathcal{C}$ and $\mathcal{D}$ be two classes of modules. In Section 3, we set the following classes of $T$-modules:
\begin{align*}
&\mathrm{Rep}(\mathcal{C}_{A},\mathcal{D}_{B})=\{N=\left(\begin{smallmatrix}  N_{1}  \\   N_{2} \\\end{smallmatrix}\right)_{\phi^{N}}\in\text{T-Mod}\mid N_{1}\in\mathcal{C}_{A},~N_{2}\in\mathcal{D}_{B} \};\\
&\Phi(\mathcal{C}_{A},\mathcal{D}_{B})=\{X=\left(\begin{smallmatrix}  X_{1}  \\   X_{2} \\\end{smallmatrix}\right)_{\phi^{X}}\in\text{T-Mod}\mid \phi^{X}\text{ is monomorphic},~\mathrm{coker}\phi^{X}\in\mathcal{C}_{A},~X_{2}\in\mathcal{D}_{B} \};\\
&\Psi(\mathcal{C}_{A},\mathcal{D}_{B})=\{Y=\left(\begin{smallmatrix}  Y_{1}  \\   Y_{2} \\\end{smallmatrix}\right)_{\psi^{Y}}\in\text{T-Mod}\mid \psi^{Y}\text{ is epimorphic},~Y_{1}\in\mathcal{C}_{A},~\mathrm{ker}\psi^{Y}\in\mathcal{D}_{B} \}.
\end{align*}

We have the following main result which shows that two complete hereditary cotorsion pairs, one in $A$-Mod and one in $B$-Mod, induce two complete hereditary cotorsion pairs in $T$-Mod. \\

 {$\mathbf{Theorem}$} Let $A$ and $B$ be two rings and $T=\left(\begin{matrix}  A & M \\  0 & B \\\end{matrix}\right)$ with $M$ an $A$-$B$-bimodule, and let $\mathcal{A},~\mathcal{B}$, $\mathcal{C}$ and $\mathcal{D}$ be classes of modules. If $\mathrm{Tor}_{1}^{B}(M,E)=0$ for any $E\in \mathcal{C}_{B}$ and $\mathrm{Ext}_{A}^{1}(M,F)=0$ for any $F\in \mathcal{B}_{A}$,  then the following are equivalent:

(1) $(\mathcal{A}_{A},\mathcal{B}_{A})$ and $(\mathcal{C}_{B},\mathcal{D}_{B})$ are hereditary cotorsion pairs each generated by a set.

(2) $(\Phi(\mathcal{A}_{A},\mathcal{C}_{B}), \mathrm{Rep}(\mathcal{B}_{A},\mathcal{D}_{B}))$ and $(\mathrm{Rep}(\mathcal{A}_{A},\mathcal{C}_{B}),\Psi(\mathcal{B}_{A},\mathcal{D}_{B}))$ are hereditary cotorsion pairs each generated by a set.\\

The proof of this theorem is inspired by \cite[Theorem A]{ho18}.

Cotorsion pairs and their relation to model structures have been the topic of much recent research. The most wonderful result in \cite{ho02}, which is now known as Hovey's correspondence, says that there is a one-to-one correspondence between abelian model structures and complete cotorsion pairs.
Hovey's correspondence makes it clear that an abelian model structure on abelian category can be represented by a triple $\mathcal{M=(Q,W,R)}$. Becker \cite{be14} and Gillespie \cite{gj14} showed that given two complete hereditary cotorsion pairs $(\mathcal{Q,\widetilde{R}})$ and $\mathcal{(\widetilde{Q},R)}$, if $\widetilde{\mathcal{Q}}\subseteq \mathcal{Q}$ (or equivalently $\widetilde{\mathcal{R}}\subseteq \mathcal{R}$) and $\widetilde{\mathcal{Q}}\cap\mathcal{R}=\mathcal{Q}\cap\widetilde{\mathcal{R}}$, then there is a thick subcategory $\mathcal{W}$ such that $\mathcal{M=(Q,W,R)}$ forms a Hovey triple. Thus, there exists the triangulated equivalence
$$\mathrm{Ho}(\mathcal{M})\overset{\sim}{\rightarrow}\underline{\mathcal{Q}\cap\mathcal{R}},$$
where $\underline{\mathcal{Q}\cap\mathcal{R}}$ denotes the stable categpry of the Frobenius category $\mathcal{Q}\cap\mathcal{R}$ (see \cite[Theorem 4.3]{gj16b}).
If $R$ is a Gorenstein ring (i.e. a left and right Noetherian ring with finite injective dimension as either left or right module over itself ), the cotorsion pairs $(\mathcal{GP}_{R},\mathcal{W}_{R})$ and $(\mathcal{P}_{R},R\text{-Mod})$ satisfy the condition $\mathcal{P}_{R}\subseteq\mathcal{GP}_{R}$ and $\mathcal{GP}_{R}\cap\mathcal{W}_{R}=\mathcal{P}_{R}$,  where $\mathcal{W}_{R}$ is the category of left $R$-modules with finite projective (injective) dimension. Hence $\mathcal{M}=(\mathcal{GP}_{R},\mathcal{W}_{R},R\text{-Mod})$ forms a hereditary abelian model structure on $R$-Mod. Then there exists a triangulated equivalence
$$\mathrm{Ho}(\mathcal{M})\overset{\sim}{\rightarrow}\underline{\mathcal{GP}_{R}}.$$
If we consider Artin algebras and finitely generated modules, Zhang proved in \cite[Theorem 3.5]{zh13} that if $T$ is a Gorenstein algebra and $_{A}M$ is projective, then there is a recollement of $\underline{\mathcal{GP}_{T}}$
relative to $\underline{\mathcal{GP}_{A}}$ and $\underline{\mathcal{GP}_{B}}$. Inspired by the above equivalences and recollement, we answer the following question in Section 4.\\

\hspace{-0.4cm}{\bf Question}
 Let $T=\left(\begin{matrix}  A & M \\  0 & B \\\end{matrix}\right)$ be an upper triangular matrix ring. Given cofibrantly generated model structures $\mathcal{M}_{A}$ and $\mathcal{M}_{B}$ on $A$-Mod and $B$-Mod respectively. Are there a cofibrantly generated model structure $\mathcal{M}_{T}$ on $T$-Mod and a recollement of $\mathrm{Ho}(\mathcal{M}_{T})$ relative to $\mathrm{Ho}(\mathcal{M}_{A})$ and $\mathrm{Ho}(\mathcal{M}_{B})$ ?\\

 Finally, we give some applications of our main results for  Gorenstein projective and Gorenstein flat model structures. It is shown that
the homotopy category of these model structures on $T$-Mod admits a recollement relative to corresponding homotopy categories.
\bigskip

\section { \bf Preliminaries }
\bigskip
Now we introduce some notations and conventions used later in the paper. For more details the reader can consult \cite{ee00,tr06} and \cite{hv99}. All the rings we consider will be associative rings with identity, all the
modules considered will be unital modules. For any ring $R$, we denote the category of left $R$-modules by $R$-Mod.\\

\hspace{-0.4cm}\textbf{2.1} \emph{Cotorsion pairs}. ~A cotorsion pair is a pair $(\mathcal{A,B})$ of classes of left $R$-modules such that $\mathcal{A}^{\perp}=\mathcal{B}$ and $A={^{\perp}\mathcal{B}}$. Here $\mathcal{A}^{\perp}$ is
the class of left $R$-modules $X$ such that
Ext$_{R}^{1}(A,X)=0$ for all $A\in \mathcal{A}$, and
similarly $^{\perp}\mathcal{B}$ is the class of  left $R$-modules $Y$ such that Ext$^{1}_{R}(Y, B)=0$ for all $B \in \mathcal{B}$.
A cotorsion pair $(\mathcal{A}, \mathcal{B})$ is said to be \emph{complete} if it has enough projectives and injectives, i.e., for any left $R$-module $X$, there are exact sequences $0\rightarrow B \rightarrow A \rightarrow X \rightarrow 0$ and $0 \rightarrow X \rightarrow B' \rightarrow A'\rightarrow 0$ respectively with $B, B'\in \mathcal{B}$ and
$A, A'\in \mathcal{A}$. A cotorsion pair $(\mathcal{A}, \mathcal{B})$ is complete if and only if  $(\mathcal{A}, \mathcal{B})$ has enough injectives if and only if $(\mathcal{A}, \mathcal{B})$ has enough projectives.

Let $\mathcal{C}$ be a class of left $R$-modules. Following \cite[Definition 5.15]{tr06}, the cotorsion pair \emph{generated by $\mathcal{C}$} is $(^{\perp}(\mathcal{C}^{\perp}),\mathcal{C}^{\perp})$ and the cotorsion pair \emph{cogenerated by $\mathcal{C}$} is $(^{\perp}\mathcal{C},(^{\perp}\mathcal{C})^{\perp})$. By \cite[Theorem 6.11]{tr06}, if a cotorsion pair $(\mathcal{A}, \mathcal{B})$  is generated by a set, then it is complete. We say that a class $\mathcal{G}$ of left $R$-modules is \emph{generating} if any left $R$-module is the quotient of a set-indexed coproduct of modules in $\mathcal{G}$.  A cotorsion pair $(\mathcal{A}, \mathcal{B})$ is called \emph{small} \cite[Definition 6.4]{ho02} if it is  generated by a set and $\mathcal{A}$ is generating.

 A class of left $R$-modules is \emph{resolving} if it contains all the projective left $R$-modules and is closed under extensions, kernels of epimorphisms and direct summands. We say that a cotorsion  pair  $(\mathcal{A}, \mathcal{B})$ is  \emph{resolving} if  $\mathcal{A}$ is resolving;  $(\mathcal{A}, \mathcal{B})$ is \emph{coresolving} if the right hand class $\mathcal{B}$ satisfies the dual; $(\mathcal{A}, \mathcal{B})$ is \emph{hereditary} \cite{tr06} if it is both resolving and coresolving.

Let $C\in R$-Mod and $\mathcal{F}$ a
class of modules closed under isomorphic images and direct summands.  An $\mathcal{F}$-\emph{precover} of $C$ is a homomorphism $\phi: F\rightarrow C$ with $F\in \mathcal{F}$ such that given any other homomorphism $\phi': F'\rightarrow C$ with $F'\in \mathcal{F}$, there exists a homomorphism $\varphi :F' \rightarrow F$ such that $\phi' = \phi\varphi$. An $\mathcal{F}$-precover $\phi: F\rightarrow C$ is called \emph{special}, if $\phi$ is epimorphic and ker$\phi\in \mathcal{F}^{\perp}$. $\mathcal{F}$-preenvelopes and special $\mathcal{F}$-preenvelopes are defined dually.\\

\hspace{-0.4cm}\textbf{2.2} \emph{Recollement.} Let $\mathcal{T}',~ \mathcal{T},~ \mathcal{T}''$ be triangulated categories. We give the definition that appeared in \cite{he05} based on localization and colocalization sequences. The standard reference is \cite{be82}.

\begin{definition} Let $\mathcal{T}'\stackrel{F}\rightarrow \mathcal{T}\stackrel{G}\rightarrow \mathcal{T}''$ be a sequence of triangulated functors between triangulated categories. We say it is a \emph{localization sequence} when there exist right adjoints $F_{\rho}$ and $G_{\rho}$ giving a diagram of functors as below with the listed properties.
$$\xymatrix{
  \mathcal{T}' \ar@<0.6ex>[r]^{F} & \mathcal{T}\ar@<0.6ex>[l]^{F_{\rho}}\ar@<0.6ex>[r]^{G} & \mathcal{T}''.\ar@<0.6ex>[l]^{G_{\rho}} }$$

(1) The right adjoint $F_{\rho}$ of $F$ satisfies $F_{\rho}\circ F=id_{\mathcal{T}'}$.

(2) The right adjoint $G_{\rho}$ of $G$ satisfies $G\circ G_{\rho} =id_{\mathcal{T}''}$.

(3) For any object $X\in \mathcal{T}$, we have $GX=0$ if and only if $X\cong FX'$ for some $X'\in \mathcal{T}'$.

A colocalization sequence is the dual. That is, there must exist left adjoints $F_{\lambda}$ and $G_{\lambda}$ with the analogous properties.
\end{definition}

This brings us to the definition of a recollement where the sequence of functors $\mathcal{T}'\stackrel{F}\rightarrow \mathcal{T}\stackrel{G}\rightarrow \mathcal{T}''$ is both a localization sequence and a colocalization sequence.

 \begin{definition} Let $\mathcal{T}'\stackrel{F}\rightarrow \mathcal{T}\stackrel{G}\rightarrow \mathcal{T}''$ be a sequence of exact functors between triangulated categories. We say $\mathcal{T}'\stackrel{F}\rightarrow \mathcal{T}\stackrel{G}\rightarrow \mathcal{T}''$ induces a \emph{recollement} if it is both a localization sequence and a colocalization sequence as shown in the picture
$$\xymatrix{\mathcal{T}'\ar[r]^{F}&\ar@<-3ex>[l]_{F_{\lambda}}\ar@<3ex>[l]^{F_{\rho}}\mathcal{T}
\ar[r]^{G}&\ar@<-3ex>[l]_{G_{\lambda}}\ar@<3ex>[l]^{G_{\rho}}\mathcal{T}''.}$$
\end{definition}
For more details of
recollements of abelian categories we refer the reader to \cite{cp14}.\\

\hspace{-0.4cm}\textbf{2.3}\emph{~ Upper triangular matrix rings.} Let $A$, $B$ be two rings and $T=\left(\begin{matrix}  A & M \\  0 & B \\\end{matrix}\right)$ with $M$ an $A$-$B$-bimodule. Next, we recall the description of left $T$-modules via column vectors. Let $X_{1}\in A$-Mod and $X_{2}\in B$-Mod, and let $\phi^{X}:M\otimes_{B} X_{2}\rightarrow X_{1}$ be a homomorphism of left $A$-modules. The left $T$-module structure on $X=\binom{X_{1}}{X_{2}}$ is defined by the following identity
$$\left(\begin{matrix}  a & m\\
0&b\\\end{matrix}\right)\left(\begin{matrix}  x_{1}\\
x_{2}\\\end{matrix}\right)=\left(\begin{matrix}  ax_{1}+\phi^{X}(m\otimes x_{2}) \\
bx_{2}\\\end{matrix}\right),$$
where $a\in A,~b\in B,~m\in M,~x_{i}\in X_{i}$ for $i=1,~2$.	
According to \cite[Theorem 1.5]{gr82}, $T$-Mod is equivalent to the category whose objects are triples $X=\binom{X_{1}}{X_{2}}_{\phi^{X}}$, where $X_{1}\in A$-$\mathrm{Mod}$, $X_{2}\in B$-$\mathrm{Mod}$ and $\phi^{X}:M\otimes_{B} X_{2}\rightarrow X_{1}$ is an $A$-homomorphism, and whose morphisms between two objects $X=\binom{X_{1}}{X_{2}}_{\phi^{X}}$ and $Y=\binom{Y_{1}}{Y_{2}}_{\phi^{Y}}$~are pairs $\binom{f_{1}}{f_{2}}$ such that $f_{1}\in \mathrm{Hom}_{A}(X_{1},Y_{1})$, $f_{2}\in \mathrm{Hom}_{B}(X_{2},Y_{2})$, satisfying that the diagram
$$\xymatrix{
  M\otimes X_{2} \ar[d]_{\phi^{X}} \ar[r]^{1_{M}\otimes_{B} f_{2}} & M\otimes_{B} Y_{2} \ar[d]_{ \phi^{Y}} \\
  X_{1} \ar[r]^{f_{1}} & Y_{1}   }
$$
is commutative. In the rest of the paper we identify $T$-Mod with this category and, whenever there is no possible confusion, we omit the homomorphism $\phi$. Consequently, throughout the paper, a left $T$-module is a pair $\binom{X_{1}}{X_{2}}$. Given such a module $X$, we denote by $\psi^{M}$ the morphism from $X_{2}$ to $\mathrm{Hom}_{A}(M,X_{1})$ given by $\psi^{M}(x)(m)=\phi^{X}( m\otimes x)$ for each $x\in X_{2},~m\in M$.

Note that a sequence of $T$-modules
$$0\rightarrow\left(\begin{matrix} M_{1}' \\   M_{2}' \\\end{matrix}\right)\rightarrow\left(\begin{matrix}  M_{1}  \\   M_{2} \\\end{matrix}\right)\rightarrow\left(\begin{matrix}  M_{1}''  \\   M_{2}'' \\\end{matrix}\right)\rightarrow0$$
is exact if and only if both sequences $0\rightarrow M_{1}'\rightarrow M_{1}\rightarrow M_{1}''\rightarrow0$ of $A$-modules
and $0\rightarrow M_{2}'\rightarrow M_{2}\rightarrow M_{2}''\rightarrow0$ of $B$-modules are exact.

\section { \bf Cotorsion pairs over upper triangular matrix rings }
\bigskip
According to the recollement constructed by \cite{cp14} and \cite{zh13}, we have the following recollement of abelian categories:
$$\xymatrix{A\text{-}\mathrm{Mod}\ar[r]^{i_{\ast}}&\ar@<-3ex>[l]_{i^{\ast}}
\ar@<3ex>[l]^{i^{!}}T\text{-}\mathrm{Mod}
\ar[r]^{j^{\ast}}&\ar@<-3ex>[l]_{j_{!}}\ar@<3ex>[l]^{j_{\ast}}B\text{-}\mathrm{Mod},}$$
where $i^{\ast}$ is given by $\left(\begin{matrix}  X  \\   Y \\\end{matrix}\right)_{\phi}\mapsto \mathrm{coker}\phi$; $i_{\ast}$ is given by $X\mapsto \left(\begin{matrix}  X  \\   0 \\\end{matrix}\right)$; $i^{!}$ is given by $\left(\begin{matrix}  X  \\   Y \\\end{matrix}\right)_{\phi}\mapsto X$; $j_{!}$ is given by $Y\mapsto \left(\begin{matrix}  M\otimes _{B}Y  \\   Y \\\end{matrix}\right)_{id}$; $j^{\ast}$ is given by $\left(\begin{matrix}  X  \\   Y \\\end{matrix}\right)_{\phi}\mapsto Y$; $j_{\ast}$ is given by $Y\mapsto \left(\begin{matrix}  0  \\   Y \\\end{matrix}\right)$.
  Note that the functor $i_{\ast}$, $i^{!}$, $j^{\ast}$, $j_{\ast}$ defined above are exact. Moreover, by \cite[Lemma 3.2]{lu17}, $i^{!}$ admits a right adjoint functor $i_{?}:A\text{-}\mathrm{Mod}\rightarrow T\text{-}\mathrm{Mod}$ given by $X\mapsto \left(\begin{matrix}  X  \\ \mathrm{Hom}_{A}(M,X) \\\end{matrix}\right)_{}$, and $j_{\ast}$ admits a right adjoint functor $j^{?}:T\text{-}\mathrm{Mod}\rightarrow B\text{-}\mathrm{Mod}$ given by $\left(\begin{matrix}  X  \\   Y \\\end{matrix}\right)_{\psi}\mapsto \mathrm{ker}\psi$.

  \begin{lemma} \label{lem3.1} (see \cite[Lemma 3.10]{lu17}) Let $R$ and $S$ be two rings, and let $F:R\text{-}\mathrm{Mod}\rightarrow S\text{-}\mathrm{Mod}$ be a functor admitting a right adjoint functor $G$. If $F$ is an exact functor and preserves projective modules, or $G$ is an exact functor and preserves injective modules, then $\mathrm{Ext}^{k}_{S}(F(X), Y)\cong \mathrm{Ext}^{k}_{R}(X, G(Y))$ for $k\geq1$.
\end{lemma}

Keeping the notations as above, we have the following lemma.
\begin{lemma}\label{lem3.2} Let $L \in A\text{-}\mathrm{Mod}$, $Y \in B\text{-}\mathrm{Mod}$, $N=\left(\begin{matrix}  N_{1}  \\   N_{2} \\\end{matrix}\right)_{\phi^{N}}$, $X=\left(\begin{matrix}  X_{1}  \\   X_{2} \\\end{matrix}\right)_{\psi^{X}}\in T\text{-}\mathrm{Mod}$. We have the following
natural isomorphisms:

(1) $\mathrm{Ext}^{1}_{T}(i_{\ast}L, N)\cong \mathrm{Ext}^{1}_{A}(L, i^{!}N)$;

(2) $\mathrm{Ext}^{1}_{B}(j^{\ast}X, Y)\cong \mathrm{Ext}^{1}_{T}(X, j_{\ast}Y)$;

(3) If $\phi^{N}$ is monomorphic, then $\mathrm{Ext}^{1}_{A}(i^{\ast}N, L)\cong \mathrm{Ext}^{1}_{T}(N, i_{\ast}L)$;

(4) If $\psi^{X}$ is epimorphic, then $\mathrm{Ext}^{1}_{T}(j_{\ast}Y, X)\cong \mathrm{Ext}^{1}_{B}(Y, j^{?}X)$;

(5) If $\mathrm{Tor}^{B}_{1}(M_{B},Y)=0$, then $\mathrm{Ext}^{1}_{T}(j_{!}Y, X)\cong \mathrm{Ext}^{1}_{B}(Y, j^{\ast}X)$;

(6) If $\mathrm{Ext}^{1}_{A}(_{A}M,L)=0$, then $\mathrm{Ext}^{1}_{A}(i^{!}N, L)\cong \mathrm{Ext}^{1}_{T}(N, i_{?}L)$.

\end{lemma}

\begin{proof}
(1) The isomorphism follows from Lemma \ref{lem3.1} and the fact that $i_{\ast}$ is an exact functor and preserves projective modules.

(2) The isomorphism follows from Lemma \ref{lem3.1} and the fact that $j_{\ast}$ is an exact functor and preserves injective modules.

(3) If $\phi^{N}$ is monomorphic, consider a short exact sequence in $T\text{-}\mathrm{Mod}$
$$0\longrightarrow\left(\begin{matrix}  K_{1}  \\   K_{2} \\\end{matrix}\right)_{\phi^{K}}\longrightarrow\left(\begin{matrix}  P_{1}  \\   P_{2} \\\end{matrix}\right)_{\phi^{P}}\longrightarrow\left(\begin{matrix}  N_{1}  \\   N_{2} \\\end{matrix}\right)_{\phi^{N}}\longrightarrow0,$$
where $\left(\begin{matrix}  P_{1}  \\   P_{2} \\\end{matrix}\right)_{\phi^{P}}$ is a projective $T$-module. We get a commutative diagram
$$\xymatrix{
  & M\otimes _{B}K_{2} \ar[d]_{\phi^{K}} \ar[r]^{} &M\otimes _{B}P_{2} \ar[d]_{\phi^{P}} \ar[r]^{} & M\otimes _{B}N_{2} \ar[d]_{\phi^{N}} \ar[r]^{} & 0\\
  0 \ar[r]^{} & K_{1} \ar[r]^{} & P_{1} \ar[r]^{} & N_{1} \ar[r]^{} & 0.   }$$
Now the Snake Lemma gives us an exact sequence $0\rightarrow \mathrm{coker}\phi^{K}\rightarrow \mathrm{coker}\phi^{P}\rightarrow \mathrm{coker}\phi^{N}\rightarrow0$. Since $P$ is projective, $\mathrm{coker}\phi^{P}$ is a projective $A$-module. Thus the isomorphism follows from the first paragraph of the proof in \cite[Lemma 3.10]{lu17}.

The proof of (4) is dual to that of (3).

(5) Consider the short exact sequence in $T$-Mod
$$0\longrightarrow\left(\begin{matrix}  X_{1}  \\0  \\\end{matrix}\right)\longrightarrow\left(\begin{matrix}  X_{1}  \\ X_{2} \\\end{matrix}\right)_{\phi^{X}}\longrightarrow\left(\begin{matrix}0\\ X_{2} \\\end{matrix}\right)\longrightarrow0.\eqno(\ast)
$$
Since the map related to $j_{!}Y=\left(\begin{matrix}  M\otimes _{B}Y  \\   Y \\\end{matrix}\right)$ is identity and  $i^{\ast}\left(\begin{matrix}  M\otimes _{B}Y  \\   Y \\\end{matrix}\right)_{id}=0$, the statement (3) above yields that

\begin{align*}
\mathrm{Ext}^{1}_{T}(j_{!}Y,\left(\begin{matrix}  X_{1}  \\0  \\\end{matrix}\right))
&=\mathrm{Ext}^{1}_{T}(\left(\begin{matrix}  M\otimes _{B}Y  \\   Y \\\end{matrix}\right)_{id},i_{\ast} X_{1} )\cong\mathrm{Ext}^{1}_{A}(i^{\ast}\left(\begin{matrix}  M\otimes _{B}Y  \\   Y \\\end{matrix}\right)_{id}, X_{1}  )=0.
\end{align*}
Next, we claim that $\mathrm{Ext}^{2}_{T}(\left(\begin{smallmatrix}  M\otimes _{B}Y  \\   Y \\\end{smallmatrix}\right),\left(\begin{smallmatrix}  X_{1}  \\0  \\\end{smallmatrix}\right))=0$. Let $0\rightarrow K\stackrel{g }\rightarrow P\stackrel{h }\rightarrow Y\rightarrow0$ be an exact sequence with $P$ projective. Applying $j_{!}$, we get an exact sequence $ j_{!}K\stackrel{ }\longrightarrow j_{!}P\stackrel{ } \longrightarrow j_{!}Y\longrightarrow0$ which can be visualised by the commutative diagram:
$$\xymatrix{
  & M\otimes _{B}K \ar[d]_{\phi} \ar[r]^{1\otimes g} & M\otimes _{B}P \ar@{=}[d]_{} \ar[r]^{1\otimes h} & M\otimes _{B}Y\ar[r]^{} \ar@{=}[d]_{} &0 \\
  0 \ar[r]^{} & \mathrm{ker}(1\otimes h) \ar[r]^{\sigma} & M\otimes _{B}P \ar[r]^{1\otimes h} & M\otimes _{B}Y \ar[r]^{} & 0,   }$$
where $\phi$ is the unique $A$-homomorphism such that $1\otimes g=\sigma\phi$. If $\mathrm{Tor}^{B}_{1}(M_{B},Y)=0$, then $1\otimes g$ is monomorphic, so $\phi$ is an isomorphism by the Five Lemma. Therefore, we get an exact sequence in $T$-Mod:
$$0\longrightarrow \binom{\mathrm{ker}(1\otimes h)}{K}_{\phi}\longrightarrow j_{!}P\stackrel{ j_{!}h} \longrightarrow j_{!}Y\longrightarrow0.$$
 Applying $\mathrm{Hom}_{T}(-,\left(\begin{smallmatrix}  X_{1}  \\0  \\\end{smallmatrix}\right))$, we get the exact sequence
 $$\cdots\rightarrow \mathrm{Ext}^{1}_{T}(\binom{\mathrm{ker}(1\otimes h)}{K}_{\phi},\binom{X_{1}}{0})\rightarrow\mathrm{Ext}^{2}_{T}(j_{!}Y,\binom{X_{1}}{0})
 \rightarrow\mathrm{Ext}^{2}_{T}(j_{!}P,\binom{X_{1}}{0}).$$
Since $j_{!}P$ is projective, we have $\mathrm{Ext}^{2}_{T}(j_{!}P,\binom{X_{1}}{0})=0$. Since  $\phi$ is an isomorphism, we also have $\mathrm{Ext}^{1}_{T}(\binom{\mathrm{ker}(1\otimes h)}{K}_{\phi},\binom{X_{1}}{0})=0$. Then it is easy to check that $\mathrm{Ext}^{2}_{T}(j_{!}Y,\left(\begin{smallmatrix}  X_{1}  \\0  \\\end{smallmatrix}\right))=0$. This proves our claim.

 Finally, applying $\mathrm{Hom}_{T}(j_{!}Y,-)$ to the exact sequence ($\ast$), we get the isomorphisms
 \begin{align*}
\mathrm{Ext}^{1}_{T}(j_{!}Y,X)
& =\mathrm{Ext}^{1}_{T}({\left(\begin{matrix}  M\otimes _{B}Y  \\   Y \\\end{matrix}\right)_{id}}, {\left(\begin{matrix}  X_{1}  \\   X_{2} \\\end{matrix}\right)_{\psi^{X}}})
 \cong \mathrm{Ext}^{1}_{T}({\left(\begin{matrix}  M\otimes _{B}Y  \\   Y \\\end{matrix}\right)_{id}}, {\left(\begin{matrix}0\\ X_{2} \\\end{matrix}\right)})\\
 &\cong\mathrm{Ext}^{1}_{B}(Y, X_{2})\cong\mathrm{Ext}^{1}_{B}(Y, j^{\ast}(X)).
 \end{align*}

The proof of (6) is dual to that of (5).
\end{proof}
For two classes $\mathcal{C}$ and $\mathcal{D}$ of modules, we set the following classes of $T$-modules:
\begin{align*}
&\mathrm{Rep}(\mathcal{C}_{A},\mathcal{D}_{B})=\{N=\left(\begin{smallmatrix}  N_{1}  \\   N_{2} \\\end{smallmatrix}\right)_{\phi^{N}}\in\text{T-Mod}\mid N_{1}\in\mathcal{C}_{A},~N_{2}\in\mathcal{D}_{B} \};\\
&\Phi(\mathcal{C}_{A},\mathcal{D}_{B})=\{X=\left(\begin{smallmatrix}  X_{1}  \\   X_{2} \\\end{smallmatrix}\right)_{\phi^{X}}\in\text{T-Mod}\mid \phi^{X}\text{ is monomorphic},~\mathrm{coker}\phi^{X}\in\mathcal{C}_{A},~X_{2}\in\mathcal{D}_{B} \};\\
&\Psi(\mathcal{C}_{A},\mathcal{D}_{B})=\{Y=\left(\begin{smallmatrix}  Y_{1}  \\   Y_{2} \\\end{smallmatrix}\right)_{\psi^{Y}}\in\text{T-Mod}\mid \psi^{Y}\text{ is epimorphic},~Y_{1}\in\mathcal{C}_{A},~\mathrm{ker}\psi^{Y}\in\mathcal{D}_{B} \}.
\end{align*}
Similarly, the notation such as $\mathrm{Rep}(\mathcal{C}_{A}^{\perp}{,\mathcal{D}_{B}^{\perp}})$ is defined by $$\mathrm{Rep}(\mathcal{C}_{A}^{\perp},{\mathcal{D}_{B}^{\perp}})=\{N=\left(\begin{smallmatrix}  N_{1}  \\   N_{2} \\\end{smallmatrix}\right)_{\phi^{N}}\in\text{T-Mod}\mid N_{1}\in\mathcal{C}_{A}^{\perp},~N_{2}\in\mathcal{D}_{B}^{\perp} \}.$$
Note that the notation such as $i_{\ast}(\mathcal{C}_{A})$ is the set $\{i_{\ast}(C)\mid C\in\mathcal{C}_{A}\}$.
If the classes $\mathcal{C}$ and $\mathcal{D}$ of modules are the same, we set
$\mathrm{Rep}(\mathcal{C}_{A},\mathcal{D}_{B})=\mathrm{Rep}(\mathcal{C})$, $\Phi(\mathcal{C}_{A},\mathcal{D}_{B})=\Phi(\mathcal{C})$ and $\Psi(\mathcal{C}_{A},\mathcal{D}_{B})=\Psi(\mathcal{C})$.
Denote by $\mathcal{I}_{A}$ (resp. $\mathcal{P}_{B}$) the class of injective left $A$-modules (resp. projective left $B$-modules). We have the following observation.

\begin{proposition} \label{pro3.3}Let $T=\left(\begin{matrix}  A & M \\  0 & B \\\end{matrix}\right)$ be an upper triangular matrix ring and $\mathcal{C}_{A}$ (resp. $\mathcal{D}_{B}$) a class of $A$-modules (resp. $B$-modules). Set
\begin{align*}
   & \mathcal{S}_{1}(\mathcal{C}_{A},\mathcal{D}_{B})=\{X\in \text{T-}\mathrm{Mod}\mid X\in i_{\ast}(\mathcal{C}_{A})~or~X\in j_{!}(\mathcal{D}_{B})\}; \\
   & \mathcal{S}_{2}(\mathcal{C}_{A},\mathcal{D}_{B})=\{X\in \text{T-}\mathrm{Mod}\mid X\in i_{?}(\mathcal{C}_{A})~or~X\in j_{\ast}(\mathcal{D}_{B})\};\\
   & \mathcal{S}(\mathcal{C}_{A},\mathcal{D}_{B})=\{X\in \text{T-}\mathrm{Mod}\mid X\in i_{\ast}(\mathcal{C}_{A})~or~X\in j_{\ast}(\mathcal{D}_{B})\}.
\end{align*}

($a$) If $\mathrm{Tor}_{1}^{B}(M_{B},P)=0$ for any $P\in \mathcal{D}_{B}$, then one has $\mathcal{S}_{1}(\mathcal{C}_{A},\mathcal{D}_{B})^{\perp}=\mathrm{Rep}(\mathcal{C}_{A}^{\perp},\mathcal{D}_{B}^{\perp})$.

($b$) If $\mathrm{Ext}_{A}^{1}(_{A}M,Q)=0$ for any $Q\in \mathcal{C}_{A}$, then one has $^{\perp}\mathcal{S}_{2}(\mathcal{C}_{A},\mathcal{D}_{B})=\mathrm{Rep}(^{\perp}\mathcal{C}_{A},{^{\perp}\mathcal{D}_{B}})$.

($c$)  If $\mathcal{I}_{A}\subseteq \mathcal{C}_{A}$, then one has $^{\perp}\mathcal{S}(\mathcal{C}_{A},\mathcal{D}_{B})=\Phi(^{\perp}\mathcal{C}_{A},{^{\perp}\mathcal{D}_{B}})$.

($d$)  If $\mathcal{P}_{B}\subseteq \mathcal{D}_{B}$, then one has $\mathcal{S}(\mathcal{C}_{A},\mathcal{D}_{B})^{\perp}=\Psi(\mathcal{C}_{A}^{\perp},\mathcal{D}_{B}^{\perp})$.

\end{proposition}
\begin{proof}
 Let $X=\left(\begin{smallmatrix}  X_{1}  \\   X_{2} \\\end{smallmatrix}\right)_{\phi^{X}}\in \mathcal{S}_{1}(\mathcal{C}_{A},\mathcal{D}_{B})^{\perp}$. Note that $X\in(i_{\ast}(\mathcal{C}_{A})\cup j_{!}(\mathcal{D}_{B}))^{\perp}=i_{\ast}(\mathcal{C}_{A})^{\perp}\cap j_{!}(\mathcal{D}_{B})^{\perp}$. Then ($a$) follows immediately from Lemma \ref{lem3.2} (1) and (5).

Let $Y=\left(\begin{smallmatrix}  Y_{1}  \\   Y_{2} \\\end{smallmatrix}\right)_{\phi^{Y}}\in {^{\perp}\mathcal{S}_{2}(\mathcal{C}_{A},\mathcal{D}_{B})}$. Note that $Y\in{^{\perp}(i_{?}(\mathcal{C}_{A})\cup j_{\ast}(\mathcal{D}_{B}))}={^{\perp}i_{?}(\mathcal{C}_{A})}\cap {^{\perp}j_{\ast}(\mathcal{D}_{B})}$. Then ($b$) follows immediately from Lemma \ref{lem3.2} (2) and (6).

For (c), if $X=\left(\begin{smallmatrix}  X_{1}  \\   X_{2} \\\end{smallmatrix}\right)_{\phi^{X}}\in\Phi(^{\perp}\mathcal{C}_{A},{^{\perp}\mathcal{D}_{B}})$, we get that $\phi^{X}$ is a monomorphism, $X_{2}\in{^{\perp}\mathcal{D}_{B}}$, $\mathrm{coker}\phi^{X}\in{^{\perp}\mathcal{C}_{A}}$. It follows from Lemma \ref{lem3.2} (2) and (3) that $X\in {^{\perp}i_{\ast}(\mathcal{C}_{A})}$ and $X\in { ^{\perp}j_{\ast}(\mathcal{D}_{B})}$, so we conclude that $X\in{^{\perp}\mathcal{S}(\mathcal{C}_{A},\mathcal{D}_{B})}$. Conversely, if $X=\left(\begin{smallmatrix}  X_{1}  \\   X_{2} \\\end{smallmatrix}\right)_{\phi^{X}}\in{^{\perp}\mathcal{S}(\mathcal{C}_{A},\mathcal{D}_{B})}$, we first show that $\mathrm{Hom}_{A}(\phi^{X},N)$ is epimorphic for each $A$-module $N\in \mathcal{C}_{A}$, that is, we must show for every homomorphism $f:M\otimes_{B}X_{2}\rightarrow N$, there exists a homomorphism $g$  making the following diagram commutative:
$$\xymatrix{
  M\otimes_{B}X_{2} \ar[d]_{f} \ar[r]^-{\phi^{X}} & X_{1} \ar@{-->}[ld]^{g}      \\
  N. }$$
Consider exact sequences
$0 \rightarrow N\stackrel{\binom{1}{0}}\rightarrow N\oplus X_{1}\stackrel{(0~~1)}\rightarrow X_{1}\rightarrow0$ and $0\rightarrow 0\rightarrow X_{2}\rightarrow X_{2}\rightarrow0$. Then we define an $A$-module homomorphism $\phi':=\binom{f}{\phi^{X}}:M\otimes_{B} X_{2}\rightarrow N\oplus X_{1} $ to get a $T$-module $\binom{N\oplus X_{1}}{X_{2}}_{\phi'}$. Since $(0~~1)\binom{f}{\phi^{X}}=\phi^{X}$, it is easy to check that the sequence

$$0\longrightarrow\left(\begin{matrix}  N  \\   0 \\\end{matrix}\right)_{0}\stackrel{\binom{\binom{1}{0}}{0}}\longrightarrow\left(\begin{matrix}  N\oplus X_{1}  \\   X_{2} \\\end{matrix}\right)_{\phi'}\stackrel{\binom{(0~1)}{1}}\longrightarrow\left(\begin{matrix}  X_{1}  \\   X_{2} \\\end{matrix}\right)_{\phi^{X}}\longrightarrow0$$
is exact in $T$-Mod. Note that $\mathrm{Ext}^{1}_{T}(X,i_{\ast}(N))=0$, hence the sequence above is split. So there exists a morphism $\binom{\alpha}{\beta}:\binom{X_{1}}{X_{2}}\rightarrow \binom{N\oplus X_{1}}{X_{2}}$ such that $\binom{(0~1)}{1}\binom{\alpha}{\beta}=id_{X}$, where $\alpha=\binom{a}{b}:X_{1}\rightarrow N\oplus X_{1}$ and $\beta: X_{2}\rightarrow X_{2}$. Hence $\beta=id_{X_{2}}$. We also have the following commutative diagram in $A$-Mod:
$$\xymatrixcolsep{4pc}\xymatrix{
  M\otimes_{B} X_{2} \ar[d]_{\binom{f}{\phi^{X}}} \ar@<0.5ex>[r]^{M\otimes 1} & M\otimes_{B} X_{2} \ar[d]^{\phi^{X}}\ar@<0.5ex>[l]^{M\otimes \beta}  \\
   N\oplus X_{1}\ar@<0.5ex>[r]^{(0~1)} & X_{1}.\ar@<0.5ex>[l]^{\alpha=\binom{a}{b}}   }$$
Since $M\otimes \beta=id_{M\otimes_{B}X_{2}},~(0~1)\binom{a}{b}=\mathrm{Id}_{X_{1}}$ and $\binom{f}{\phi^{X}}(M\otimes \beta)=\binom{a}{b}\phi^{X}$, we have $b=\mathrm{Id}_{X_{1}}$ and $\binom{f}{\phi^{X}}=\binom{a\phi^{X}}{\phi^{X}}$. So $f=a\phi^{X}$. If we put $g:=a$, then $f=g\phi^{X}$. By assumption, since $A$-Mod has injective cogenerators, $\phi^{X}$ is a monomorphism. On the other side, the isomorphisms from Lemma \ref{lem3.2} (2) and (3) imply that $\mathrm{coker}\phi^{X}\in{^{\perp}\mathcal{C}_{A}},~X_{2}\in{^{\perp}\mathcal{D}_{B}}$. Thus $X\in\Phi(^{\perp}\mathcal{C}_{A},{^{\perp}\mathcal{D}_{B}})$.

Similarly, ($d$) follows from Lemma \ref{lem3.2} (4) and (6).
\end{proof}
We construct four cotorsion pairs in $T$-Mod in the next theorem.
\begin{theorem} \label{the3.4}Let $\mathcal{A},~\mathcal{B}$, $\mathcal{C}$ and $\mathcal{D}$ be classes of modules. Suppose $(\mathcal{A}_{A},\mathcal{B}_{A})$ is a cotorsion pair in $A$-$\mathrm{Mod}$ generated by a class $\mathcal{M}_{A}$ and cogenerated by a class $\mathcal{N}_{A}$ $($e.g., $\mathcal{M}_{A}=\mathcal{A}_{A},~\mathcal{N}_{A}=\mathcal{B}_{A})$,
$(\mathcal{C}_{B},\mathcal{D}_{B})$  is a cotorsion pair in $B$-$\mathrm{Mod}$  generated by a class $\mathcal{U}_{B}$ and cogenerated by a class $\mathcal{V}_{B}$ $($e.g., $\mathcal{U}_{B}=\mathcal{C}_{B},~\mathcal{V}_{B}=\mathcal{D}_{B})$.

($a$) If $\mathrm{Tor}_{1}^{B}(M_{B},P)=0$ for any $P\in \mathcal{U}_{B}$, then the cotorsion pair in $T$-$\mathrm{Mod}$ generated by $\mathcal{S}_{1}(\mathcal{M}_{A},\mathcal{U}_{B})$ is  $$\mathfrak{C}_{1}=(^{\perp}\mathrm{Rep}(\mathcal{B}_{A},\mathcal{D}_{B}),\mathrm{Rep}(\mathcal{B}_{A},\mathcal{D}_{B})).$$

If $\mathcal{I}_{A}\subseteq \mathcal{N}_{A}$, then the cotorsion pair in $T$-$\mathrm{Mod}$ cogenerated by $\mathcal{S}(\mathcal{N}_{A},\mathcal{V}_{B})$ is
$$\mathfrak{C}_{2}=(\Phi(\mathcal{A}_{A},\mathcal{C}_{B}),\Phi(\mathcal{A}_{A},\mathcal{C}_{B})^{\perp}).$$

($b$)  If $\mathrm{Ext}_{A}^{1}(_{A}M,Q)=0$ for any $Q\in \mathcal{N}_{A}$, then the cotorsion pair in $T$-$\mathrm{Mod}$ cogenerated by $\mathcal{S}_{2}(\mathcal{N}_{A},\mathcal{V}_{B})$ is  $$\mathfrak{C}_{3}=(\mathrm{Rep}(\mathcal{A}_{A},\mathcal{C}_{B}),\mathrm{Rep}(\mathcal{A}_{A},\mathcal{C}_{B})^{\perp}).$$

If $\mathcal{P}_{B}\subseteq \mathcal{U}_{B}$, then the cotorsion pair in $T$-$\mathrm{Mod}$ generated by $\mathcal{S}(\mathcal{M}_{A},\mathcal{U}_{B})$ is
$$\mathfrak{C}_{4}=(^{\perp}\Psi(\mathcal{B}_{A},\mathcal{D}_{B}),\Psi(\mathcal{B}_{A},\mathcal{D}_{B})).$$
\end{theorem}
\begin{proof} Part $(a)$ follows from Proposition \ref{pro3.3} $(a),~(c)$, and $(b)$ follows from Proposition \ref{pro3.3} $(b),~(d)$.
\end{proof}
\begin{remark}\label{rem3.5} Note that if $\mathcal{M}_{A}$ and $\mathcal{U}_{B}$ are sets, then so are $\mathcal{S}_{1}(\mathcal{M}_{A},\mathcal{U}_{B})$ and $\mathcal{S}(\mathcal{M}_{A},\mathcal{U}_{B})$. Thus, if each of the cotorsion pairs $(\mathcal{A}_{A},\mathcal{B}_{A})$ and $(\mathcal{C}_{B},\mathcal{D}_{B}) $ is generated by a set, then so are $(^{\perp}\mathrm{Rep}(\mathcal{B}_{A},\mathcal{D}_{B}),
\mathrm{Rep}(\mathcal{B}_{A},\mathcal{D}_{B}))$ and $(^{\perp}\Psi(\mathcal{B}_{A},\mathcal{D}_{B}),\Psi(\mathcal{B}_{A},\mathcal{D}_{B}))$.
\end{remark}

Suppose that $(\mathcal{A}_{A},\mathcal{B}_{A})$ and $(\mathcal{C}_{B},\mathcal{D}_{B})$ are hereditary, it is natural to ask if the induced cotorsion pairs in Theorem \ref{the3.4} have the same property.

\begin{proposition}\label{pro3.6}Let $T=\left(\begin{matrix}  A & M \\  0 & B \\\end{matrix}\right)$ be an upper triangular matrix ring. Adopt the notations and assumptions from Theorem \ref{the3.4}.
If the cotorsion pairs $(\mathcal{A}_{A},\mathcal{B}_{A})$ and $(\mathcal{C}_{B},\mathcal{D}_{B}) $ are hereditary, then so are cotorsion pairs $\mathfrak{C}_{1}$ and $\mathfrak{C}_{3}$. Furthermore,

(1)  $\mathrm{Tor}_{1}^{B}(M,E)=0$ for any $E\in \mathcal{C}_{B}$ if and only if the cotorsion pair $\mathfrak{C}_{2}$ is hereditary.

(2) $\mathrm{Ext}_{A}^{1}(M,F)=0$ for any $F\in \mathcal{B}_{A}$ if and only if the cotorsion pair $\mathfrak{C}_{4}$ is hereditary.

\end{proposition}
\begin{proof}

By \cite[Lemma 5.24]{tr06}, we only need to show that if $\mathcal{A}_{A}$ and $\mathcal{C}_{B}$ are resolving, then so are $\mathrm{Rep}(\mathcal{A}_{A},\mathcal{C}_{B})$ and $\Phi(\mathcal{A}_{A},\mathcal{C}_{B})$, and if $\mathcal{B}_{A}$ and $\mathcal{D}_{B}$ are coresolving, then so are $\mathrm{Rep}(\mathcal{B}_{A},\mathcal{D}_{B})$ and $\Psi(\mathcal{B}_{A},\mathcal{D}_{B})$.

It is clear that cotorsion pairs $\mathfrak{C}_{1}$ and $\mathfrak{C}_{3}$ are hereditary.

(1) Suppose that $\mathrm{Tor}_{1}^{B}(M_{B},E)=0$ for any $E\in \mathcal{C}_{B}$. To see $\Phi(\mathcal{A}_{A},\mathcal{C}_{B})$ is resolving, note that $\Phi(\mathcal{A}_{A},\mathcal{C}_{B})$ is closed under extensions and contains all projective modules in $T$-Mod since $\Phi(\mathcal{A}_{A},\mathcal{C}_{B})$ is the left half of a cotorsion pair. It remains to see that if $0\rightarrow X\rightarrow Y\rightarrow Z\rightarrow0$ is a short exact sequence in $T$-Mod with $Y=\binom{Y_{1}}{Y_{2}}_{\phi^{Y}},~Z=\binom{Z_{1}}{Z_{2}}_{\phi^{Z}}\in\Phi(\mathcal{A}_{A},\mathcal{C}_{B})$, then one also has $X=\binom{X_{1}}{X_{2}}_{\phi^{X}}\in\Phi(\mathcal{A}_{A},\mathcal{C}_{B})$. To this end, consider the following commutative diagram
$$\xymatrix{
  & M\otimes _{B}X_{2} \ar[d]_{\phi^{X}} \ar[r]^{} &M\otimes _{B}Y_{2} \ar[d]_{\phi^{Y}} \ar[r]^{} & M\otimes _{B}Z_{2} \ar[d]_{\phi^{Z}} \ar[r]^{} & 0\\
  0 \ar[r]^{} & X_{1} \ar[r]^{} & Y_{1} \ar[r]^{} & Z_{1} \ar[r]^{} & 0.   }$$
By assumption, $\mathrm{Tor}_{1}^{B}(M_{B},Z_{2})=0$, then $\phi^{X}$ is a monomorphism. From the Snake Lemma and the assumption, it now follows that coker$\phi^{X}$ is in $\mathcal{A}_{A}$. Since $\mathcal{C}_{B}$ is resolving, $X_{2}\in\mathcal{C}_{B}$. We conclude that $X\in \Phi(\mathcal{A}_{A},\mathcal{C}_{B})$.

Conversely, assume that $(\Phi(\mathcal{A}_{A},\mathcal{C}_{B}),\Phi(\mathcal{A}_{A},\mathcal{C}_{B})^{\perp})$ is hereditary. Then for each $E\in \mathcal{C}_{B}$, there is a $T$-module $\binom{M\otimes_{B}E}{E}\in \Phi(\mathcal{A}_{A},\mathcal{C}_{B})$. Consider the following exact sequence
$$0\longrightarrow\left(\begin{matrix}  K_{1}  \\   K_{2} \\\end{matrix}\right)_{\phi^{K}}\stackrel{\binom{f}{g}}\longrightarrow\left(\begin{matrix}  P_{1}  \\   P_{2} \\\end{matrix}\right)_{\phi^{P}}\longrightarrow\left(\begin{matrix}  M\otimes_{B}E  \\   E \\\end{matrix}\right)_{\mathrm{Id}}\longrightarrow0,$$
 where $\binom{P_{1}}{P_{2}}\in \Phi(\mathcal{A}_{A},\mathcal{C}_{B})$ is a projective $T$-module. Since $\Phi(\mathcal{A}_{A},\mathcal{C}_{B})$ is resolving, we have $\binom{K_{1}}{K_{2}}_{\phi^{K}}\in \Phi(\mathcal{A}_{A},\mathcal{C}_{B})$ and $\phi^{K}$ is monomorphic. Thus, we get a commutative diagram
$$\xymatrix{
  & M\otimes _{B}K_{2} \ar[d]_{\phi^{K}} \ar[r]^{1\otimes g} &M\otimes _{B}P_{2} \ar[d]_{\phi^{P}} \ar[r]^{} & M\otimes _{B}E \ar@{=}[d]_{} \ar[r]^{} & 0\\
  0 \ar[r]^{} & K_{1} \ar[r]^{f} & P_{1} \ar[r]^{} & M\otimes_{B}E \ar[r]^{} & 0,   }$$
where $1\otimes g$ is monomorphic. It is now very easy to check that $\mathrm{Tor}_{1}^{B}(M_{B},E)=0$, finishing the proof.

(2) The proof is dual to that of (1).
\end{proof}

\begin{proposition}\label{pro3.7}Let $T=\left(\begin{matrix}  A & M \\  0 & B \\\end{matrix}\right)$ be an upper triangular matrix ring. Adopt the notations from Theorem \ref{the3.4}.

(1) If $\mathrm{Tor}_{1}^{B}(M_{B},E)=0$ for any $E\in \mathcal{C}_{B}$, then
$$(^{\perp}\mathrm{Rep}(\mathcal{B}_{A},\mathcal{D}_{B}),\mathrm{Rep}(\mathcal{B}_{A},\mathcal{D}_{B}))=(\Phi(\mathcal{A}_{A},\mathcal{C}_{B}),\Phi(\mathcal{A}_{A},\mathcal{C}_{B})^{\perp}).$$

(2) If $\mathrm{Ext}_{A}^{1}(_{A}M,F)=0$ for any $F\in \mathcal{B}_{A}$, then
$$(\mathrm{Rep}(\mathcal{A}_{A},\mathcal{C}_{B}),\mathrm{Rep}(\mathcal{A}_{A},\mathcal{C}_{B})^{\perp})=(^{\perp}\Psi(\mathcal{B}_{A},\mathcal{D}_{B}),\Psi(\mathcal{B}_{A},\mathcal{D}_{B})).$$

\end{proposition}
\begin{proof} We just prove (1) since (2) follows by duality. From Theorem \ref{the3.4} we have
$$\mathrm{Rep}(\mathcal{B}_{A},\mathcal{D}_{B})=\mathcal{S}_{1}(\mathcal{A}_{A},\mathcal{C}_{B})^{\perp}~~\mathrm{and}~~\Phi(\mathcal{A}_{A},\mathcal{C}_{B})={^{\perp}\mathcal{S}(\mathcal{B}_{A},\mathcal{D}_{B})}$$
and it must be shown that $\mathrm{Rep}(\mathcal{B}_{A},\mathcal{D}_{B})=\Phi(\mathcal{A}_{A},\mathcal{C}_{B})^{\perp}$. For all objects $X\in\mathcal{A}_{A}$, $X'\in\mathcal{C}_{B}$ and $Y\in\mathcal{B}_{A}$, we have
$\mathrm{Ext}^{1}_{T}(i_{\ast}(X),i_{\ast}(Y))\cong\mathrm{Ext}^{1}_{A}(X,i^{!}i_{\ast}(Y))\cong\mathrm{Ext}^{1}_{A}(X,Y)=0$ and $\mathrm{Ext}^{1}_{T}(j_{!}(X'),i_{\ast}(Y))\cong\mathrm{Ext}^{1}_{A}(X',j^{\ast}i_{\ast}(Y))=0$. Thus there is  the inclusion $\mathcal{S}_{1}(\mathcal{A}_{A},\mathcal{C}_{B})=i_{\ast}(\mathcal{A}_{A})\cup j_{!}(\mathcal{C}_{B})\subseteq {^{\perp}i_{\ast}(\mathcal{B}_{A})}$. For all objects $Y'\in\mathcal{D}_{B}$, we have $\mathrm{Ext}^{1}_{T}(i_{\ast}(X),j_{\ast}(Y'))\cong\mathrm{Ext}^{1}_{A}(X,i^{!}j_{\ast}(Y'))=0$
and $\mathrm{Ext}^{1}_{T}(j_{!}(X'),j_{\ast}(Y'))\cong\mathrm{Ext}^{1}_{B}(X',j^{\ast}j_{\ast}(Y'))\cong\mathrm{Ext}^{1}_{B}(X',Y')=0$. So we have the inclusion $\mathcal{S}_{1}(\mathcal{A}_{A},\mathcal{C}_{B})\subseteq {^{\perp}j_{\ast}(\mathcal{D}_{B})}$. Consequently, $\mathrm{Rep}(\mathcal{B}_{A},\mathcal{D}_{B})={\mathcal{S}_{1}(\mathcal{A}_{A},\mathcal{C}_{B})^{\perp}}
\supseteq(^{\perp}i_{\ast}(\mathcal{B}_{A})\cap{^{\perp}j_{\ast}(\mathcal{D}_{B})})^{\perp}=
(^{\perp}(\mathcal{S}(\mathcal{B}_{A},\mathcal{D}_{B})))^{\perp}=\Phi(\mathcal{A}_{A},\mathcal{C}_{B})^{\perp}$.

To show the opposite inclusion, it suffices to argue that every $X=\binom{X_{1}}{X_{2}}\in \mathrm{Rep}(\mathcal{B}_{A},\mathcal{D}_{B})$ is in $\mathcal{S}(\mathcal{B}_{A},\mathcal{D}_{B})$. Consider the short exact sequence
$$0\longrightarrow\left(\begin{matrix}  X_{1} \\   0 \\\end{matrix}\right)_{}\longrightarrow\left(\begin{matrix}   X_{1}  \\   X_{2} \\\end{matrix}\right)_{}\longrightarrow\left(\begin{matrix}  0  \\   X_{2} \\\end{matrix}\right)_{}\longrightarrow0.$$
Since $\binom{X_{1}}{0}=i_{\ast}(X_{1})\subseteq i_{\ast}(\mathcal{B}_{A})\subseteq\mathcal{S}(\mathcal{B}_{A},\mathcal{D}_{B})\subseteq(^{\perp}(\mathcal{S}(\mathcal{B}_{A},\mathcal{D}_{B})))^{\perp}=\Phi(\mathcal{A}_{A},\mathcal{C}_{B})^{\perp}$ and $\binom{0}{X_{2}}=j_{\ast}(X_{2})\subseteq\mathcal{S}(\mathcal{B}_{A},\mathcal{D}_{B})
\subseteq\Phi(\mathcal{A}_{A},\mathcal{C}_{B})^{\perp}$, we get that $\binom{X_{1}}{X_{2}}\subseteq\Phi(\mathcal{A}_{A},\mathcal{C}_{B})^{\perp}$ since $\Phi(\mathcal{A}_{A},\mathcal{C}_{B})^{\perp}$ is closed under extensions. Hence  $\mathrm{Rep}(\mathcal{B}_{A},\mathcal{D}_{B})=\Phi(\mathcal{A}_{A},\mathcal{C}_{B})^{\perp}$.
\end{proof}

The following is the main result of this section which provides a way to construct hereditary complete cotorsion pairs in $T$-Mod.

\begin{theorem}\label{the3.8}Let $A$ and $B$ be two rings and $T=\left(\begin{matrix}  A & M \\  0 & B \\\end{matrix}\right)$ with $M$ an $A$-$B$-bimodule, and let $\mathcal{A},~\mathcal{B}$, $\mathcal{C}$ and $\mathcal{D}$ be classes of modules. If $\mathrm{Tor}_{1}^{B}(M,E)=0$ for any $E\in \mathcal{C}_{B}$ and $\mathrm{Ext}_{A}^{1}(M,F)=0$ for any $F\in \mathcal{B}_{A}$,  then the following are equivalent:

(1) $(\mathcal{A}_{A},\mathcal{B}_{A})$ and $(\mathcal{C}_{B},\mathcal{D}_{B})$ are hereditary cotorsion pairs each generated by a set.

(2) $(\Phi(\mathcal{A}_{A},\mathcal{C}_{B}), \mathrm{Rep}(\mathcal{B}_{A},\mathcal{D}_{B}))$ and $(\mathrm{Rep}(\mathcal{A}_{A},\mathcal{C}_{B}),\Psi(\mathcal{B}_{A},\mathcal{D}_{B}))$ are hereditary cotorsion pairs each generated by a set.
\end{theorem}
\begin{proof} (1) $\Rightarrow$ (2) Suppose $(\mathcal{C}_{B},\mathcal{D}_{B})$ is generated by a set $\mathcal{U}_{B}$, we can add a projective generator $B$ such that $(\mathcal{U}_{B}\cup B)^{\perp}=\mathcal{D}_{B}$. Then $(\mathcal{C}_{B},\mathcal{D}_{B})$ is generated by a set containing a  projective generator. Therefore, Theorem \ref{the3.4} and Proposition \ref{pro3.7} show that $(\Phi(\mathcal{A}_{A},\mathcal{C}_{B}), \mathrm{Rep}(\mathcal{B}_{A},\mathcal{D}_{B}))$ and $(\mathrm{Rep}(\mathcal{A}_{A},\mathcal{C}_{B}),\Psi(\mathcal{B}_{A},\mathcal{D}_{B}))$ are cotorsion pairs in $T$-Mod.  These two cotorsion pairs are each generated by a set by Theorem \ref{the3.4} and Remark \ref{rem3.5}, and they are hereditary by Proposition \ref{pro3.6}.

(2) $\Rightarrow$ (1) First, we show $(\mathcal{A}_{A},\mathcal{B}_{A})$ is a cotorsion pair. Let $X\in\mathcal{A}_{A}, Y\in\mathcal{B}_{A}$. By Lemma \ref{lem3.2}, we have isomorphisms $\mathrm{Ext}^{1}_{A}(X,Y)=\mathrm{Ext}^{1}_{A}(X, i^{!}\binom{Y}{0})\cong\mathrm{Ext}^{1}_{T}(i_{\ast}X,\binom{Y}{0})$. Note that $i_{\ast}X\in\Phi(\mathcal{A}_{A},\mathcal{C}_{B}),~ \binom{Y}{0}\in \mathrm{Rep}(\mathcal{B}_{A},\mathcal{D}_{B})$, we get that $\mathrm{Ext}^{1}_{A}(X,Y)=0$.
Let $X\in {^{\perp}\mathcal{B}_{A}}$, i.e., $\mathrm{Ext}^{1}_{A}(X,Y)=0$ for any $Y\in \mathcal{B}_{A}$. We'll show that $X\in \mathcal{A}_{A}$. By Lemma \ref{lem3.2} again, $0=\mathrm{Ext}^{1}_{A}(X,Y)=\mathrm{Ext}^{1}_{A}(i^{!}\binom{X}{0},Y)
\cong\mathrm{Ext}^{1}_{T}(\binom{X}{0},i_{?}Y)$.
Note that $i_{?}Y=\binom{Y}{\mathrm{Hom}_{A}(M,Y)}\in \Psi(\mathcal{B}_{A},\mathcal{D}_{B})$, we get that $\binom{X}{0}\in \mathrm{Rep}(\mathcal{A}_{A},\mathcal{C}_{B})$. So $X\in \mathcal{A}_{A}$.
Let $Y\in \mathcal{A}_{A}^{\perp}$, i.e., $\mathrm{Ext}^{1}_{A}(X,Y)=0$ for any $X\in \mathcal{A}_{A}$. We'll show that $Y\in \mathcal{B}_{A}$. By Lemma \ref{lem3.2} again, $0=\mathrm{Ext}^{1}_{A}(X,Y)\cong\mathrm{Ext}^{1}_{A}(i_{\ast}X,\binom{Y}{0})$.
Note that $i_{\ast}X=\binom{X}{0}\in \Phi(\mathcal{A}_{A},\mathcal{C}_{B})$, we get $\binom{Y}{0}\in \mathrm{Rep}(\mathcal{B}_{A},\mathcal{D}_{B})$. So $Y\in \mathcal{B}_{A}$. Thus $(\mathcal{A}_{A},\mathcal{B}_{A})$ is a cotorsion pair.

Suppose $(\Phi(\mathcal{A}_{A},\mathcal{C}_{B}), \mathrm{Rep}(\mathcal{B}_{A},\mathcal{D}_{B}))$ is generated by a set $\mathcal{S'}$, then we have $\mathcal{S'}\subseteq\Phi(\mathcal{A}_{A},\mathcal{C}_{B})$ and $\mathcal{S'}^{\perp}=\mathrm{Rep}(\mathcal{B}_{A},\mathcal{D}_{B})$. By the isomorphism from Lemma \ref{lem3.2} (3), we get $i^{\ast}(\mathcal{S'})\subseteq\mathcal{A}_{A}$ and $(i^{\ast}(\mathcal{S'}))^{\perp}=\mathcal{B}_{A}$. It follows that $(\mathcal{A}_{A},\mathcal{B}_{A})$ is generated by a set.

In order to prove that $(\mathcal{A}_{A},\mathcal{B}_{A})$  is hereditary, we only need to show that $\mathcal{A}_{A}$ is closed under kernels of epimorphisms. Consider the exact sequence
$0\rightarrow X\rightarrow Y\rightarrow Z\rightarrow0$  with $Y,~Z\in \mathcal{A}_{A}$. Applying the functor $i_{\ast}$, we get an exact sequence
$0\rightarrow i_{\ast}X\rightarrow i_{\ast}Y\rightarrow i_{\ast}Z\rightarrow0$ in $T$-Mod, where $i_{\ast}Y,~i_{\ast}Z\in\Phi(\mathcal{A}_{A},\mathcal{C}_{B})$. Since $\Phi(\mathcal{A}_{A},\mathcal{C}_{B})$ is resolving, we have  $i_{\ast}X=\binom{X}{0}\in\Phi(\mathcal{A}_{A},\mathcal{C}_{B})$. Thus $X\in \mathcal{A}_{A}$.

The method used to show $(\mathcal{C}_{B},\mathcal{D}_{B})$ is a hereditary cotorsion pair  generated by a set is dual.
\end{proof}

Specializing Theorem \ref{the3.8} to the case $\mathcal{A=C}$ and $\mathcal{B=D}$, we have the following result.

\begin{corollary}\label{cor3.9}Let $A$ and $B$ be two rings and $T=\left(\begin{matrix}  A & M \\  0 & B \\\end{matrix}\right)$ with $M$ an $A$-$B$-bimodule, and let $\mathcal{A}$ and $\mathcal{B}$ be classes of modules. Suppose $(\mathcal{A},\mathcal{B})$ is a hereditary cotorsion pair generated by a set.

(1) If $\mathrm{Tor}_{1}^{B}(M,E)=0$ for any $E\in \mathcal{A}_{B}$, then
  $(\Phi(\mathcal{A}), \mathrm{Rep}(\mathcal{B}))$ is a hereditary cotorsion pair generated by a set.

(2) If $\mathrm{Ext}_{A}^{1}(M,F)=0$ for any $F\in \mathcal{B}_{A}$, then
 $(\mathrm{Rep}(\mathcal{A}),\Psi(\mathcal{B}))$ is a hereditary cotorsion pair generated by a set.
\end{corollary}

\bigskip
\section { \bf Recollements of homotopy categories }
\leftskip0truemm \rightskip0truemm
A nice introduction to the basic idea of a model category can be found in \cite{hv99}.
We begin by recalling Hovey's correspondence. First, we need the definition of an abelian model structure.

\begin{definition}\label{def4.1}An abelian model category \cite{hv99} is a bicomplete abelian category $\mathcal{C}$ equipped with a model structure such that:

(1) a map is a cofibration if and only if it is a monomorphism with cofibrant cokernel,

(2) a map is a fibration if and only if it is an epimorphism with fibrant kernel.
\end{definition}

 Hovey  then  characterizes  abelian model categories in  terms
of cotorsion pairs. So in fact one could even take the cotorsion pairs given in the correspondence below as the definition of an abelian model category.
By  a  thick  subcategory of  an  abelian category $\mathcal{C}$ we  mean  a  class  $\mathcal{W}$ of  objects of $\mathcal{C}$
which is closed under direct summands and such that if two out of three of the terms in a short exact sequence are in $\mathcal{W}$, then so is the third.\\

\hspace{-0.4cm}\textbf{Hovey's correspondence }(see \cite[Theorem 2.2]{ho02})
\emph{ Let $\mathcal{C}$ be  an  abelian  category  with  an  abelian  model  structure.  Let $\mathcal{Q}$ be
the  class  of  cofibrant  objects, $\mathcal{R}$ the  class  of  fibrant  objects  and  $\mathcal{W}$ the  class  of  trivial
objects.  Then $\mathcal{W}$ is  a  thick  subcategory  of $\mathcal{C}$ and  both  $(\mathcal{Q}, \mathcal{W} \cap \mathcal{R})$ and  $(\mathcal{Q}\cap \mathcal{W} , \mathcal{R})$ are
complete  cotorsion  pairs  in $\mathcal{C}$. Conversely,  given  a  thick  subcategory  $\mathcal{W}$ and  classes  $\mathcal{Q}$
and  $\mathcal{R}$ making  $(\mathcal{Q}, \mathcal{W} \cap \mathcal{R})$ and  $(\mathcal{Q}\cap \mathcal{W} , \mathcal{R})$ each  a complete  cotorsion  pair,  then  there
is  an  abelian  model  structure  on  $\mathcal{C}$  where  $\mathcal{Q}$ is the class of   the  cofibrant  objects,  $\mathcal{R}$ is the class of  the  fibrant
objects  and   $\mathcal{W}$  is the class of  the  trivial  objects.
}
\\

Let  $\mathcal{Q}$, $\mathcal{W}$, and $\mathcal{R}$ be  three  classes  in  $\mathcal{C}$ as in Hovey's correspondence.
Then  we  call  $\mathcal{M}=(\mathcal{Q}, \mathcal{W} , \mathcal{R)}$ a  \emph{Hovey  triple}. By a \emph{hereditary Hovey triple} we mean that the two corresponding cotorsion pairs $(\mathcal{Q}, \mathcal{W} \cap \mathcal{R})$ and  $(\mathcal{Q}\cap \mathcal{W} , \mathcal{R})$ are hereditary. Let $W$ be the class of weak equivalences. The homotopy category of the model category is the localization $\mathcal{C}[W^{-1}]$ and is denoted by $\mathrm{Ho}(\mathcal{M})$.
 We know that if $\mathcal{M=(Q,W,R)}$ is a hereditary Hovey triple, then $\mathrm{Ho}(\mathcal{M})$ is a triangulated category and it is triangle equivalent to the stable category $(\mathcal{Q}\cap \mathcal{R})/\omega$, where $\omega:=\mathcal{Q}\cap\mathcal{W}\cap \mathcal{R}$ is the class of projective-injective objects. \\

The following lemma provides us a way to construct a Hovey triple from two cotorsion pairs.

\begin{lemma} (see \cite[Proposition 1.4.2]{be14}, \cite[Theorem 1.1]{gj14})\label{lem4.2} Let $\mathcal{C}$ be an abelian category and suppose $(\mathcal{Q},\widetilde{\mathcal{R}})$ and $(\widetilde{\mathcal{Q}},\mathcal{R})$ are complete (small), hereditary cotorsion pairs over $\mathcal{C}$ with (1) $\widetilde{\mathcal{Q}}\subseteq \mathcal{Q}$, (2) $\mathcal{Q}\cap\widetilde{\mathcal{R}}=\widetilde{\mathcal{Q}}\cap\mathcal{R}$. Then there exists a unique (cofibrantly generated) abelian model
structure $(\mathcal{Q}, \mathcal{W} , \mathcal{R})$, and its class $\mathcal{W}$ of trivial objects is given by
\begin{align*}
 \mathcal{W }&=\{ X \in\mathcal{C} \mid \exists~\text{a short exact sequence}~ 0\rightarrow X\rightarrow R\rightarrow Q\rightarrow0~with~R\in\widetilde{\mathcal{R}},~Q\in\widetilde{\mathcal{Q}} \}\\
&=\{ X \in\mathcal{C} \mid \exists~\text{a short exact sequence}~ 0\rightarrow Q'\rightarrow R'\rightarrow X\rightarrow0~with~R'\in\widetilde{\mathcal{R}},~Q'\in\widetilde{\mathcal{Q}}\}.
\end{align*}
\end{lemma}

By the lemma above, given two hereditary small cotorsion pairs $(\mathcal{A}_{A},\mathcal{B}_{A})$, $(\mathcal{C}_{A},\mathcal{D}_{A})$ in $A$-Mod, which satisfy the conditions (1) and (2) in Lemma 4.2, there exists a cofibrantly generated abelian model structure $\mathcal{M}_{A}=(\mathcal{A}_{A},\mathcal{W}_{A},\mathcal{D}_{A})$. Similarly, given two hereditary small cotorsion pairs $(\widetilde{\mathcal{A}}_{B},\widetilde{\mathcal{B}}_{B})$, $(\widetilde{\mathcal{C}}_{B},\widetilde{\mathcal{D}}_{B})$ in $B$-Mod, which satisfy the conditions (1) and (2) in Lemma \ref{lem4.2}, there exists a unique class $\widetilde{\mathcal{W}}_{B}$, such that $\mathcal{M}_{B}=(\widetilde{\mathcal{A}}_{B},\widetilde{\mathcal{W}}_{B},\widetilde{\mathcal{D}}_{B})$ is a cofibrantly generated abelian model structure. Therefore, if $\mathrm{Tor}_{1}^{B}(M,X)=0$ for any $X\in \widetilde{\mathcal{A}}_{B}$, then by the proof of Theorem \ref{the3.8}, we have two complete hereditary cotorsion pairs $(\Phi(\mathcal{A}_{A},\widetilde{\mathcal{A}}_{B}), \mathrm{Rep}(\mathcal{B}_{A},\widetilde{\mathcal{B}}_{B}))$ and $(\Phi(\mathcal{C}_{A},\widetilde{\mathcal{C}}_{B}), \mathrm{Rep}(\mathcal{D}_{A},\widetilde{\mathcal{D}}_{B}))$. In general, these two cotorsion pairs don't satisfy the conditions (1) and (2) in Lemma \ref{lem4.2}. So we introduce the concept of perfect bimodules. A bimodule $_{A}M_{B}$ is \emph{perfect relative to $\mathcal{M}_{A}$ and $\mathcal{M}_{B}$}, if the induced cotorsion pairs $(\Phi(\mathcal{A}_{A},\widetilde{\mathcal{A}}_{B}), \mathrm{Rep}(\mathcal{B}_{A},\widetilde{\mathcal{B}}_{B}))$ and $(\Phi(\mathcal{C}_{A},\widetilde{\mathcal{C}}_{B}), \mathrm{Rep}(\mathcal{D}_{A},\widetilde{\mathcal{D}}_{B}))$ have the same core, that is, $\Phi(\mathcal{A}_{A},\widetilde{\mathcal{A}}_{B})\cap \mathrm{Rep}(\mathcal{B}_{A},\widetilde{\mathcal{B}}_{B})=\Phi(\mathcal{C}_{A},\widetilde{\mathcal{C}}_{B})\cap \mathrm{Rep}(\mathcal{D}_{A},\widetilde{\mathcal{D}}_{B})$.
For example, $M$ is perfect relative to $\mathcal{M}_{A}$ and $\mathcal{M}_{B}$ whenever $M\otimes_{B}\widetilde{\mathcal{B}}_{B}\subseteq \mathcal{B}_{A}$ and $M\otimes_{B}\widetilde{\mathcal{C}}_{B}\subseteq \mathcal{C}_{A}$. If $M$ is perfect, by Lemma \ref{lem4.2} again, there exists a unique class $\mathcal{W}_{T}$, such that  $\mathcal{M}_{T}=(\Phi(\mathcal{A}_{A},\widetilde{\mathcal{A}}_{B}),\mathcal{W}_{T},\mathrm{Rep}(\mathcal{D}_{A},\widetilde{\mathcal{D}}_{B}))$ is a Hovey triple. \\

A Quillen map of model categories $\mathcal{M} \rightarrow \mathcal{N}$ consists of a pair
of adjoint functors $(L, R) : \mathcal{M} \rightleftarrows \mathcal{N}$ such that $L$ preserves cofibrations and trivial
cofibrations (it is equivalent to require that $R$ preserves fibrations and trivial fibrations). In this case the pair $(L,R)$ is also called a \emph{Quillen adjunction}. A Quillen map induces adjoint total derived functors between the homotopy
categories \cite{ma12}. The class of weak equivalences is the most important class of morphisms in a model category. It is easy to
see that a map is a weak equivalence if and only if it factors as a trivial cofibration followed by a trivial fibration.  The following important characterization
is proved in \cite[Lemma 5.8]{ho02}.

\begin{lemma}\label{lem4.3} Let $\mathcal{M= (Q,W,R)}$ be a Hovey triple. Then a morphism is a weak
equivalence if and only if it factors as a monomorphism with trivial cokernel followed by an
epimorphism with trivial kernel.
\end{lemma}

Before giving our main result, we need the following crucial result.
By \cite[Proposition 1.2.7]{be14}, if $\mathcal{M=(C,W,F)}$ is an abelian model structure on  a Grothendieck category $\mathcal{A}$, then $\mathcal{M}$ is cofibrantly generated if and only if $(\mathcal{C}\cap \mathcal{W},\mathcal{F})$ and $(\mathcal{C},\mathcal{W}\cap \mathcal{F})$ are small. Moreover, \cite[Corollary 1.2.2]{be14} tells us that any cotorsion pair in module categories generated by a set is small.

\begin{proposition}\label{pro4.4} Let $T=\left(\begin{matrix}  A & M \\  0 & B \\\end{matrix}\right)$ be an upper triangular matrix ring, and let $\mathcal{M}_{A} =(\mathcal{A}_{A},\mathcal{W}_{A}$,
$\mathcal{D}_{A})$ and $\mathcal{M}_{B}=(\widetilde{\mathcal{A}}_{B},\widetilde{\mathcal{W}}_{B},\widetilde{\mathcal{D}}_{B})$ be the cofibrantly generated abelian model structures on $A$-$\mathrm{Mod}$ and $B$-$\mathrm{Mod}$, respectively.  If $\mathrm{Tor}_{1}^{B}(M,X)=0$ for any $X\in \widetilde{\mathcal{A}}_{B}$ and $M$ is perfect relative to $\mathcal{M}_{A}$ and $\mathcal{M}_{B}$, then $\mathcal{M}_{T}=(\Phi(\mathcal{A}_{A},\widetilde{\mathcal{A}}_{B}),\mathcal{W}_{T},\mathrm{Rep}(\mathcal{D}_{A},\widetilde{\mathcal{D}}_{B}))$ is a cofibrantly generated abelian model structure on $T$-Mod and the sequence
$$\xymatrix{
  \mathrm{Ho}(\mathcal{M}_{A}) \ar@<0.6ex>[r]^{\mathrm{L~i_{\ast}}} & \mathrm{Ho}(\mathcal{M}_{T})\ar@<0.6ex>[l]^{\mathrm{R~i^{!}}}\ar@<0.6ex>[r]^{\mathrm{L~j^{\ast}}} & \mathrm{Ho}(\mathcal{M}_{B})\ar@<0.6ex>[l]^{\mathrm{R~j_{\ast}}} }$$
is a localization sequence of triangulated categories, where $\mathrm{L~i_{\ast}}$, $\mathrm{L~j^{\ast}}$, $\mathrm{R~i^{!}}$ and $\mathrm{R~j_{\ast}}$ are the total derived functors of those in Section 3.
\end{proposition}
\begin{proof}
Let $\mathcal{B}_{A}:=\mathcal{W}_{A}\cap\mathcal{D}_{A}$, $\mathcal{C}_{A}:=\mathcal{A}_{A}\cap\mathcal{W}_{A}$, $\widetilde{\mathcal{B}}_{B}:=\widetilde{\mathcal{W}}_{B}\cap\widetilde{\mathcal{D}}_{B}$ and $\widetilde{\mathcal{C}}_{B}:=\widetilde{\mathcal{A}}_{B}\cap\widetilde{\mathcal{W}}_{B}$.
If $\mathrm{Tor}_{1}^{B}(M,X)=0$ for any $X\in \widetilde{\mathcal{A}}_{B}$
and $M$ is perfect relative to $\mathcal{M}_{A}$ and $\mathcal{M}_{B}$,  then $\mathrm{Tor}_{1}^{B}(M,Y)=0$ for any $Y\in \widetilde{\mathcal{C}}_{B}$ and $\mathcal{M}_{T}=(\Phi(\mathcal{A}_{A},\widetilde{\mathcal{A}}_{B}),\mathcal{W}_{T},\mathrm{Rep}(\mathcal{D}_{A},\widetilde{\mathcal{D}}_{B}))$ is a cofibrantly generated abelian model structure on $T$-Mod by Theorem \ref{the3.8}.
We first claim that $(i_{\ast},i^{!})$ and $(j^{\ast},j_{\ast})$ are Quillen adjunctions. Since (trivial) cofibrations equal monomorphisms with (trivial) cofibrant cokernels and (trival) fibrations equal epimorphisms with (trivial) fibrant kernels, the inclusions $i_{\ast}(\mathcal{A}_{A})\subseteq \Phi(\mathcal{A}_{A},\widetilde{\mathcal{A}}_{B})$ and $i_{\ast}(\mathcal{A}_{A}\cap \mathcal{W}_{A})\subseteq \Phi(\mathcal{A}_{A},\widetilde{\mathcal{A}}_{B})\cap \mathcal{W}_{T}$ imply that $i_{\ast}$ preserves cofibrations and trivial cofibrations. Thus $(i_{\ast},i^{!})$ is a Quillen adjunction.
Similarly, $j^{\ast}$ is a left adjoint and preserves cofibrations and trivial cofibrations. Hence $(j^{\ast},j_{\ast})$ is a Quillen adjunction by the definition.
By \cite[Proposition 16.2.2]{ma12}, the total derived functor $\mathrm{L~i_{\ast}}$ and $\mathrm{R~i^{!}}$ exist and form an adjoint between $\mathrm{Ho}(\mathcal{M}_{A})$ and $\mathrm{Ho}(\mathcal{M}_{T})$,  $\mathrm{L~j^{\ast}}$ and $\mathrm{R~j_{\ast}}$ exist and form an adjoint between $\mathrm{Ho}(\mathcal{M}_{T})$ and $\mathrm{Ho}(\mathcal{M}_{B})$.
That is, we have the following diagram
$$\xymatrix{
  \mathrm{Ho}(\mathcal{M}_{A}) \ar@<0.6ex>[r]^{\mathrm{L~i_{\ast}}} & \mathrm{Ho}(\mathcal{M}_{T})\ar@<0.6ex>[l]^{\mathrm{R~i^{!}}}\ar@<0.6ex>[r]^{\mathrm{L~j^{\ast}}} & \mathrm{Ho}(\mathcal{M}_{B}).\ar@<0.6ex>[l]^{\mathrm{R~j_{\ast}}} }$$
In general, the right derived functor is defined on objects by first taking a fibrant replacement and then applying the functor. Similarly, the left derived functor is defined by first taking a cofibrant replacement and then applying the functor. So we have computed $(\mathrm{L~i_{\ast}},\mathrm{R~i^{!}})=(i_{\ast}Q_{A},i^{!}R_{T})$ and $(\mathrm{L~j^{\ast}},\mathrm{R~j_{\ast}})=(j^{\ast}Q_{T},j_{\ast}R_{B})$. Here, the notation such as $Q_{A}$ means to take a special $\mathcal{A}_{A}$-precover. Similarly the notation $R_{T}$ means to take a special $\mathrm{Rep}(\mathcal{D}_{A},\widetilde{\mathcal{D}}_{B})$-preenvelope.

We wish to show that these functors preserve
exact triangles. By  \cite[6.7]{ke2}, it
suffices to prove that $\mathrm{L~i_{\ast}}$  and $\mathrm{L~j^{\ast}}$ are triangulated. Recall from \cite[Proposition 4.4 and Section 5]{gj11} that the distinguished triangles in Ho$(\mathcal{M})$
are, up to isomorphism, the images in Ho$(\mathcal{M})$ of distinguished triangles in $(\mathcal{Q}\cap \mathcal{R})/\omega$ under
the equivalence $(\mathcal{Q}\cap \mathcal{R})/\omega\rightarrow \mathrm{Ho}(\mathcal{M})$. Now every exact triangle is by definition isomorphic to a standard
triangle of the form $X\stackrel{f}\rightarrow Y\rightarrow C_{f}\rightarrow \Sigma X$, arising from a
short exact sequence $0\rightarrow X\stackrel{f}\rightarrow Y\rightarrow Z\rightarrow0$ as in the pushout diagram below where $W$
is a projective-injective object:
$$\xymatrix{
    & 0 \ar[d]_{} & 0 \ar[d]_{}  &  &  \\
  0  \ar[r]^{} & X \ar[d]_{} \ar[r]^{} & Y \ar[d]_{} \ar[r]^{} & Z \ar@{=}[d]_{} \ar[r]^{} & 0  \\
  0  \ar[r]^{} & W \ar[d]_{} \ar[r]^{} & C_{f} \ar[d]_{} \ar[r]^{} & Z \ar[r]^{} & 0  \\
    & \Sigma X \ar[d]_{} \ar@{=}[r]^{} & \Sigma X \ar[d]_{}&  &  \\
   & 0 & 0.  &  &    }$$
First we point out that if $\mathcal{(Q, W, R)}$ is any hereditary Hovey triple, then the fully exact subcategory $\mathcal{Q} \cap \mathcal{R}$ is a Frobenius category with $\mathcal{Q} \cap \mathcal{W}\cap \mathcal{R}$ being precisely the class of projective-injective objects.
The main point is that by using \cite[Lemma 1.4.4]{be14} and its dual, along with the fact that both functors $i_{\ast}$ and $j^{\ast}$ are exact and preserve projective-injective objects, any diagram of this form is sent to another diagram of this form. It follows that  $\mathrm{L~i_{\ast}}$  and $\mathrm{L~j^{\ast}}$ send standard triangles (resp. exact triangles) to standard triangles (resp. exact triangles).

In order to show we have a localization sequence, it remains to show

(1) $\mathrm{R~i^{!}}\circ\mathrm{L~i_{\ast}}\cong 1_{\mathrm{Ho}(\mathcal{M}_{A})}$.

(2) $\mathrm{L~j^{\ast}}\circ\mathrm{R~j_{\ast}}\cong 1_{\mathrm{Ho}(\mathcal{M}_{B})}$.

(3) The essential image of $\mathrm{L~i_{\ast}}$ equals the kernel of $\mathrm{L~j^{\ast}}$.

To prove (1), let $f:X\rightarrow Y$ be a homomorphism in $A$-Mod.
Using the completeness of the cotorsion pair $(\mathcal{A}_{A}\cap\mathcal{W}_{A},\mathcal{D}_{A})$, we get the following commutative diagram
$$\xymatrix{
  0  \ar[r]^{} & X \ar[d]_{f} \ar[r]^{q_{A}} & X' \ar[d]_{\widetilde{f}} \ar[r]^{} & Z_{1} \ar[r]^{} & 0 \\
  0 \ar[r]^{} & Y \ar[r]^{q'_{A}} & Y' \ar[r]^{} & Z'_{1} \ar[r]^{} & 0,   }$$
  where $X',~Y'\in \mathcal{D}_{A}$, $Z_{1},~Z'_{1}\in \mathcal{A}_{A}\cap\mathcal{W}_{A}$. Note that $X'$ is a fibrant replacement of $X$ in $\mathcal{M}_{A}$, so we have natural isomorphisms $q_{A}$ and $q'_{A}$ in $\mathrm{Ho}(\mathcal{M}_{A})$. The functor $i_{\ast}Q_{A}$ acts by $\widetilde{f}\mapsto \widehat{f}$, where $\widehat{f}$ is any map making the diagram below commute
$$\xymatrix{
  0   \ar[r]^{ } & i_{\ast}K  \ar[r]^{} & i_{\ast}F_{1}\ar[d]_{\widehat{f}} \ar[r]^{j_{A}} & i_{\ast}X' \ar[d]_{i_{\ast}\widetilde{f}} \ar[r]^{} & 0  \\
  0 \ar[r]^{} & i_{\ast}K' \ar[r]^{} & i_{\ast}F'_{1} \ar[r]^{j'_{A}} & i_{\ast}Y' \ar[r]^{} & 0,   }$$
 where the rows are exact, $F_{1},~F'_{1}\in\mathcal{A}_{A}$, and $K,~K'\in \mathcal{D}_{A}\cap \mathcal{W}_{A}=\mathcal{B}_{A}$. Moreover, we get $F_{1},~F'_{1}\in\mathcal{A}_{A}\cap \mathcal{D}_{A}$ since $\mathcal{D}_{A}$ is closed under extensions. Now applying $i^{!}R$ to $\widehat{f}$ gives us $\overline{f}$ in the next commutative diagram
$$\xymatrix{
  0  \ar[r]^{} & i^{!}i_{\ast}F_{1} \ar[d]_{i^{!}\widehat{f}} \ar[r]^{p_{A}} & i^{!}L_{1} \ar[d]_{\overline{f}} \ar[r]^{} & i^{!}C_{1} \ar[r]^{} & 0 \\
  0 \ar[r]^{} & i^{!}i_{\ast}F'_{1} \ar[r]^{p'_{A}} & i^{!}L'_{1} \ar[r]^{} & i^{!}C'_{1} \ar[r]^{} & 0,   }$$
where $L_{1}, L'_{1}\in \mathrm{Rep}(\mathcal{D}_{A},\widetilde{\mathcal{D}}_{B})$, $C_{1}, C'_{1}\in \Phi(\mathcal{A}_{A},\widetilde{\mathcal{A}}_{B})\cap \mathcal{W}_{T}=\Phi(\mathcal{C}_{A},\widetilde{\mathcal{C}}_{B})$. Since $i_{\ast}F_{1},~i_{\ast}F'_{1}\in \mathrm{Rep}(\mathcal{D}_{A},\widetilde{\mathcal{D}}_{B})$ and $\mathrm{Rep}(\mathcal{D}_{A},\widetilde{\mathcal{D}}_{B})$ is coresolving, we get that $C_{1}, C'_{1}\in\Phi(\mathcal{C}_{A},\widetilde{\mathcal{C}}_{B})\cap\mathrm{Rep}(\mathcal{D}_{A},\widetilde{\mathcal{D}}_{B})=
\Phi(\mathcal{A}_{A},\widetilde{\mathcal{A}}_{B})\cap\mathrm{Rep}(\mathcal{B}_{A},\widetilde{\mathcal{B}}_{B})$ by hypotheses. Furthermore, it is easy to check we have inclusions $i_{\ast}(\mathcal{B}_{A})\subseteq \mathrm{Rep}(\mathcal{B}_{A},\widetilde{\mathcal{B}}_{B})=\mathcal{W}_{T}\cap \mathrm{Rep}(\mathcal{D}_{A},\widetilde{\mathcal{D}}_{B})$, $i^{!}(\mathrm{Rep}(\mathcal{B}_{A},\widetilde{\mathcal{B}}_{B}))\subseteq \mathcal{B}_{A}=\mathcal{W}_{A}\cap\mathcal{D}_{A}$ and $i^{!}(\Phi(\mathcal{A}_{A},\widetilde{\mathcal{A}}_{B})\cap\mathrm{Rep}(\mathcal{B}_{A},\widetilde{\mathcal{B}}_{B}))\subseteq \mathcal{B}_{A}\subseteq\mathcal{W}_{A}$. Thus $i^{!}j_{A},~i^{!}j'_{A},~p_{A},~p'_{A}$ are all weak equivalences in $\mathcal{M}_{A}$ by Lemma \ref{lem4.3}. So, in $\mathrm{Ho}(\mathcal{M}_{A})$, we have a commutative diagram
$$\xymatrix{
 X\ar[r]^{q_{A}}\ar[d]_{f}& X' \ar[d]_{\widetilde{f}} \ar[r]^{\eta_{X'}} & i^{!}i_{\ast}X' \ar[d]_{ i^{!}i_{\ast}\widetilde{f}} & i^{!}i_{\ast}F_{1} \ar[r]^{p_{A}} \ar[l]_{i^{!}j_{A}} \ar[d]^{i_{\ast}\widehat{f}} & i^{!}L_{1}\ar[d]^{\overline{f}} \\
Y\ar[r]^{q'_{A}}&  Y' \ar[r]^{\eta_{Y'}} & i^{!}i_{\ast}Y' & i^{!}i_{\ast}F'_{1} \ar[l]^{i^{!}j'_{A}}\ar[r]^{p'_{A}} & i^{!}L_{2},  }$$
where $\eta:1\rightarrow i^{!}i_{\ast}$ is the unit of adjunction $(i_{\ast}, i^{!})$. This diagram gives rise to a natural isomorphism: $\mathrm{R~i^{!}}\circ\mathrm{L~i_{\ast}}\cong 1_{\mathrm{Ho}(\mathcal{M}_{A})}.$

Next we prove (2).  Let $X\in \mathrm{Ho}(\mathcal{M}_{B})$ be any object. Using the completeness of the cotorsion pair $(\widetilde{\mathcal{A}}_{B}\cap\widetilde{\mathcal{W}}_{B},\widetilde{\mathcal{D}}_{B})$, we get a short exact sequence $0\rightarrow X\rightarrow E\rightarrow L\rightarrow0$ with $E\in\widetilde{\mathcal{D}}_{B}$ and $L\in \widetilde{\mathcal{A}}_{B}\cap\widetilde{\mathcal{W}}_{B}$. Note that $E$ is a fibrant replacement of $X$ in $\mathcal{M}_{B}$, so we have a natural isomorphism $X\cong E$ in $\mathrm{Ho}(\mathcal{M}_{B})$. The functor $\mathrm{R~j_{\ast}}=j_{\ast}R_{B}$ acts by $E\mapsto j_{\ast}D$, where $j_{\ast}D$ is in the short exact sequence $0\rightarrow j_{\ast}E\overset{p}{\rightarrow} j_{\ast}D\rightarrow j_{\ast}C\rightarrow0$. Here $D\in\widetilde{\mathcal{D}}_{B}$, $C\in\widetilde{\mathcal{A}}_{B}\cap\widetilde{\mathcal{W}}_{B}=\widetilde{\mathcal{C}}_{B}$. Since $\widetilde{\mathcal{D}}_{B}$ is coresolving, we get $C\in \widetilde{\mathcal{C}}_{B}\cap\widetilde{\mathcal{D}}_{B}=\widetilde{\mathcal{A}}_{B}\cap\widetilde{\mathcal{B}}_{B}$. Now applying $j^{\ast}Q_{T}$ to $j_{\ast}D$ gives us $j^{\ast}N$ in the next exact sequence
$$0\rightarrow j^{\ast}K\rightarrow j^{\ast}N\overset{q}{\rightarrow} j^{\ast}j_{\ast}D\rightarrow0,$$
where $N$ is a cofibrant replacement of $j_{\ast}D$, $N\in\Phi(\mathcal{A}_{A},\widetilde{\mathcal{A}}_{B})$, $K\in \mathrm{Rep}(\mathcal{B}_{A},\widetilde{\mathcal{B}}_{B})=\mathcal{W}_{T}\cap\mathrm{Rep}(\mathcal{D}_{A},\widetilde{\mathcal{D}}_{B})$. We have inclusions $j_{\ast}(\widetilde{\mathcal{A}}_{B}\cap\widetilde{\mathcal{B}}_{B})\subseteq\mathrm{Rep}(\mathcal{B}_{A},\widetilde{\mathcal{B}}_{B})=\mathcal{W}_{T}\cap\mathrm{Rep}(\mathcal{D}_{A},\widetilde{\mathcal{D}}_{B})$ and $j^{\ast}(\mathrm{Rep}(\mathcal{B}_{A},\widetilde{\mathcal{B}}_{B}))\subseteq \widetilde{\mathcal{B}}_{B}=\widetilde{\mathcal{W}}_{B}\cap\widetilde{\mathcal{D}}_{B}$. Thus $j^{\ast}p,~q$ are weak equivalences in $\mathcal{M}_{B}$. Hence, we have isomorphisms $\mathrm{L~j^{\ast}}\circ\mathrm{R~j_{\ast}}(X)\cong\mathrm{L~j^{\ast}}
\circ\mathrm{R~j_{\ast}}(E)\cong j^{\ast}N\stackrel{q}\cong j^{\ast}j_{\ast}D\stackrel{j^{\ast}p}\cong j^{\ast}j_{\ast}E\cong j^{\ast}j_{\ast}X \cong X$ in $\mathrm{Ho}(\mathcal{M}_{B})$. By an argument similar to that in (1), we see that these isomorphisms are natural.

For (3), denote by $\mathrm{Im}~\mathrm{L~i_{\ast}}$ the essential image of $\mathrm{L~i_{\ast}}$. It is easy to see $\mathrm{Im}~\mathrm{L~i_{\ast}}\subseteq \mathrm{ker}~\mathrm{L~j^{\ast}}$. Conversely, let $X=\binom{X_{1}}{X_{2}}\in \mathrm{ker}~\mathrm{L~j^{\ast}}$. We claim that there exists $Y\in A$-Mod  such that $\mathrm{L~i_{\ast}}~(Y)\cong X$ in $\mathrm{Ho}(\mathcal{M})$. By the action of the functor $\mathrm{L~j^{\ast}}$, we have an exact sequence $0\rightarrow K\rightarrow P\rightarrow X\rightarrow0$ in $T$-Mod, where $P=\binom{P_{1}}{P_{2}}_{\phi^{P}}\in\Phi(\mathcal{A}_{A},\widetilde{\mathcal{A}}_{B})$, $K=\binom{K_{1}}{K_{2}}\in\mathcal{W}_{T}\cap\mathrm{Rep}(\mathcal{D}_{A},\widetilde{\mathcal{D}}_{B})$. Note that $P$ is isomorphic to $X$ in $\mathrm{Ho}(\mathcal{M}_{T})$ and this sequence induces an exact sequence $0\rightarrow K_{2}\rightarrow P_{2}\rightarrow X_{2}\rightarrow0$ in $B$-Mod, where $K_{2}\in\widetilde{\mathcal{B}}_{B}=\widetilde{\mathcal{W}}_{B}\cap\widetilde{\mathcal{D}}_{B}$, $P_{2}\in\widetilde{\mathcal{A}}_{B}$. Since $\mathrm{L~j^{\ast}}~(X) \in \widetilde{\mathcal{W}}_{B}$, we get $P_{2}\in\widetilde{\mathcal{A}}_{B}\cap\widetilde{\mathcal{W}}_{B}=\widetilde{\mathcal{C}}_{B}$, which implies that $\binom{M\otimes_{B}P_{2}}{P_{2}}_{id}\in \Phi(\mathcal{C}_{A},\widetilde{\mathcal{C}}_{B})$.
In fact, define $Y:= \mathrm{coker}\phi^{P}$, we have $\mathrm{L~i_{\ast}}~(\mathrm{coker}\phi^{P})\cong X$ in Ho($\mathcal{M}_{T}$). Indeed, consider the short exact sequence
$$0\longrightarrow\left(\begin{matrix}  M\otimes_{B}P_{2}  \\ P_{2}  \\\end{matrix}\right)_{id}\longrightarrow\left(\begin{matrix}  P_{1}  \\ P_{2} \\\end{matrix}\right)_{\phi^{P}}\longrightarrow\left(\begin{matrix}  \mathrm{coker}\phi^{P}  \\ 0 \\\end{matrix}\right)_{0}\longrightarrow0.
$$
Since $\binom{M\otimes_{B}P_{2}}{P_{2}}\in \Phi(\mathcal{C}_{A},\widetilde{\mathcal{C}}_{B})=\mathcal{W}_{T}\cap\Phi(\mathcal{A}_{A},\widetilde{\mathcal{A}}_{B})$ and $\mathrm{coker}\phi^{P}\in\mathcal{A}_{A}$, we get that $X\cong P\cong\binom{\mathrm{coker}\phi^{P}}{0}\cong\mathrm{L~i_{\ast}}~(\mathrm{coker}\phi^{P})$ in $\mathrm{Ho}(\mathcal{M}_{T})$. Hence the desired result follows immediately.
\end{proof}
By a similar type of argument to the above proposition, we have the next result.

\begin{proposition} \label{pro4.5}Let $T=\left(\begin{matrix}  A & M \\  0 & B \\\end{matrix}\right)$ be an upper triangular matrix ring, and let $\mathcal{M}_{A}=(\mathcal{A}_{A},\mathcal{W}_{A},\mathcal{D}_{A})$ and $\mathcal{M}_{B}=(\widetilde{\mathcal{A}}_{B},\widetilde{\mathcal{W}}_{B},\widetilde{\mathcal{D}}_{B})$ be the cofibrantly generated abelian model structures on $A$-$\mathrm{Mod}$ and $B$-$\mathrm{Mod}$, respectively.  If $\mathrm{Tor}_{1}^{B}(M,X)=0$ for any $X\in \widetilde{\mathcal{A}}_{B}$ and $M$ is perfect relative to $\mathcal{M}_{A}$ and $\mathcal{M}_{B}$, then $\mathcal{M}_{T}=(\Phi(\mathcal{A}_{A},\widetilde{\mathcal{A}}_{B}),\mathcal{W}_{T},\mathrm{Rep}(\mathcal{D}_{A},\widetilde{\mathcal{D}}_{B}))$ is a cofibrantly generated abelian model structure on $T$-Mod and the sequence
$$\xymatrix{
  \mathrm{Ho}(\mathcal{M}_{A}) \ar@<-0.6ex>[r]_-{\mathrm{R~i_{\ast}}} & \mathrm{Ho}(\mathcal{M}_{T})\ar@<-0.6ex>[l]_-{\mathrm{L~i^{\ast}}}\ar@<-0.6ex>[r]_-{\mathrm{R~j^{\ast}}} & \mathrm{Ho}(\mathcal{M}_{B})\ar@<-0.6ex>[l]_-{\mathrm{L~j_{!}}} }$$
is a colocalization sequence of triangulated categories, where $\mathrm{L~i^{\ast}}$, $\mathrm{L~j_{!}}$, $\mathrm{R~i_{\ast}}$ and $\mathrm{R~j^{\ast}}$ are the total derived functors of those in Section 3.
\end{proposition}
\begin{proof} Note that the functors $i^{\ast}$ and $j_{!}$ are not exact, so we give an outline of the proof. Let $\mathcal{B}_{A}:=\mathcal{W}_{A}\cap\mathcal{D}_{A}$, $\mathcal{C}_{A}:=\mathcal{A}_{A}\cap\mathcal{W}_{A}$, $\widetilde{\mathcal{B}}_{B}:=\widetilde{\mathcal{W}}_{B}\cap\widetilde{\mathcal{D}}_{B}$ and $\widetilde{\mathcal{C}}_{B}:=\widetilde{\mathcal{A}}_{B}\cap\widetilde{\mathcal{W}}_{B}$.

First, it is easy to check $(i^{\ast},i_{\ast})$ and $(j_{!},j^{\ast})$ are Quillen adjunctions. Hence, their total derived functors exist. By using \cite[Lemma 1.4.4]{be14} and its dual, along with the fact that the functors $i_{\ast}$ and $j^{\ast}$ are exact and preserve projective-injective objects, we get that $\mathrm{R~i_{\ast}}$ and $\mathrm{R~j^{\ast}}$ are triangle functors. Then $\mathrm{L~i^{\ast}}$ and $\mathrm{L~j_{!}}$ are also triangle functors by \cite[6.4]{ke2}. To show we have a colocalization sequence, it remains to show

(1) $\mathrm{L~i^{\ast}}\circ\mathrm{R~i_{\ast}}\cong 1_{\mathrm{Ho}(\mathcal{M}_{A})}$.

(2) $\mathrm{R~j_{\ast}}\circ\mathrm{L~j_{!}}\cong 1_{\mathrm{Ho}(\mathcal{M}_{B})}$.

(3) The essential image of $\mathrm{R~i_{\ast}}$ equals the kernel of $\mathrm{R~j^{\ast}}$.

To prove (1), let $X\in \mathrm{Ho}(\mathcal{M}_{A})$. Since $(i^{\ast},i_{\ast})$ is a Quillen adjunction, $i_{\ast}$ takes trivial fibrations between fibrant object in $\mathcal{M}_{A}$ to weak equivalences in $\mathcal{M}_{T}$. By Ken Brown's Lemma (see \cite[Lemma 14.2.9]{ma12}), we get that $i_{\ast}$ takes all weak equivalences between fibrant objects to weak equivalences in $\mathcal{M}_{T}$. Thus $i_{\ast}R_{A}X\cong i_{\ast}R_{A} Q_{A} X$ in Ho$(\mathcal{M}_{T})$. Replacing $X$ by a cofibration object $X':=Q_{A}X$, we have an isomorphism $X\cong X'$ in $\mathrm{Ho}(\mathcal{M}_{A})$, where $X'\in \mathcal{A}_{A}$. The functor $\mathrm{R~i_{\ast}}=i_{\ast}R_{A}$ acts by $X'\mapsto i_{\ast}D$, where $i_{\ast}D$ is in the short exact sequence $0\rightarrow i_{\ast}X'\overset{p}{\rightarrow} i_{\ast}D\rightarrow i_{\ast}C\rightarrow0$. Here $D\in\mathcal{D}_{A}\cap\mathcal{A}_{A}$, $C\in\mathcal{A}_{A}\cap\mathcal{W}_{A}=\mathcal{C}_{A}$. Now applying $i^{\ast}R_{A}$ to $i_{\ast}D$ gives us $i_{\ast}N$ in the next exact sequence
$$ i^{\ast}K\rightarrow i^{\ast}N\overset{q}{\rightarrow} i^{\ast}i_{\ast}D\rightarrow0, \eqno(\ddag)$$
where $N$ is a cofibrant replacement of $i_{\ast}D$, $N=\binom{N_{1}}{N_{2}}\in\Phi(\mathcal{A}_{A},\widetilde{\mathcal{A}}_{B})$, $K=\binom{K_{1}}{K_{2}}\in \mathrm{Rep}(\mathcal{B}_{A},\widetilde{\mathcal{B}}_{B})\cap\Phi(\mathcal{A}_{A},\widetilde{\mathcal{A}}_{B})=\mathrm{Rep}(\mathcal{D}_{A},\widetilde{\mathcal{D}}_{B})\cap\Phi(\mathcal{C}_{A},\widetilde{\mathcal{C}}_{B})$. Consider the following commutative diagram
$$\xymatrix{
  & M\otimes _{B}K_{2} \ar[d]_{} \ar[r]^{} &M\otimes _{B}N_{2} \ar[d]_{} \ar[r]^{} & 0 \ar[d]_{} & \\
  0 \ar[r]^{} & K_{1} \ar[r]^{} & N_{1} \ar[r]^{} & D \ar[r]^{} & 0.   }$$
From the Snake Lemma, the sequence ($\ddag$) is also left exact. Furthermore,
We have inclusions $i_{\ast}(\mathcal{C}_{A})\subseteq\Phi(\mathcal{C}_{A},\widetilde{\mathcal{C}}_{B})=\mathcal{W}_{T}\cap\Phi(\mathcal{A}_{A},\widetilde{\mathcal{A}}_{B})$ and $i^{\ast}(\mathrm{Rep}(\mathcal{D}_{A},\widetilde{\mathcal{D}}_{B})\cap\Phi(\mathcal{C}_{A},\widetilde{\mathcal{C}}_{B}))\subseteq i^{\ast}(\Phi(\mathcal{C}_{A},\widetilde{\mathcal{C}}_{B})) \subseteq \mathcal{C}_{A}=\mathcal{W}_{A}\cap\mathcal{A}_{A}$. Thus $i^{\ast}p,~q$ are weak equivalences in $\mathcal{M}_{A}$. Hence, we have isomorphisms $\mathrm{L~j^{\ast}}\circ\mathrm{R~j_{\ast}}(X)\cong\mathrm{L~j^{\ast}}
\circ\mathrm{R~j_{\ast}}(X')\cong i^{\ast}N\stackrel{q}\cong i^{\ast}i_{\ast}D\stackrel{i^{\ast}p}\cong i^{\ast}i_{\ast}X'\cong X' \cong X$ in $\mathrm{Ho}(\mathcal{M}_{A})$. We see that these isomorphisms are natural.

Next we prove (2).  Let $X\in \mathrm{Ho}(\mathcal{M}_{B})$. Replacing $X$ by a cofibration object $X'$, we have an isomorphism $X\cong X'$ in $\mathrm{Ho}(\mathcal{M}_{B})$, where $X'\in \widetilde{\mathcal{A}}_{B}$. The functor $\mathrm{L~j_{!}}=j_{!}R_{B}$ acts by $X'\mapsto j_{!}F$, where $j_{!}F$ is in the short exact sequence
$$ j_{!}K\stackrel{ j_{!}g}\longrightarrow j_{!}F\stackrel{ j_{!}h} \longrightarrow j_{!}X'\longrightarrow0,\eqno(\S)$$
where $g: K\rightarrow F$ is a kernel of $h: F\rightarrow X'$, and $F\in\widetilde{\mathcal{D}}_{B}\cap\widetilde{\mathcal{A}}_{B}$. Since $\widetilde{\mathcal{A}}_{B}$ is resolving, $K\in\widetilde{\mathcal{B}}_{B}\cap\widetilde{\mathcal{A}}_{B}$.
By the proof of Lemma \ref{lem3.2} (5),  we get an exact sequence in $T$-Mod:
$$0\longrightarrow \binom{\mathrm{ker}(1\otimes h)}{K}_{\phi}\longrightarrow j_{!}F\stackrel{ j_{!}h} \longrightarrow j_{!}X'\longrightarrow0,$$
where  $j_{!}F,~j_{!}X'\in \Phi(\mathcal{A}_{A},\widetilde{\mathcal{A}}_{B})$. Then we have $\binom{\mathrm{ker}(1\otimes h)}{K}_{\phi}\in\Phi(\mathcal{A}_{A},\widetilde{\mathcal{A}}_{B})$ since $\Phi(\mathcal{A}_{A},\widetilde{\mathcal{A}}_{B})$ is resolving. It follows that $\phi$ is monomorphic. Thus $1\otimes g$ is monomorphic. So the sequence ($\S$) is also left exact.
 Now applying $j^{\ast}R$ to $j_{!}F$ gives us $j^{\ast}N$ in the next exact sequence
$$ 0\rightarrow j^{\ast}j_{!}F\stackrel{q}\rightarrow j^{\ast}N\rightarrow j^{\ast}C\rightarrow0. $$
Here $N\in\mathrm{Rep}(\mathcal{D}_{A},\widetilde{\mathcal{D}}_{B})$, $C\in\Phi(\mathcal{C}_{A},\widetilde{\mathcal{C}}_{B})$. Furthermore,
we have inclusions $j_{!}(\widetilde{\mathcal{A}}_{B}\cap\widetilde{\mathcal{B}}_{B})=j_{!}(\widetilde{\mathcal{C}}_{B}\cap
\widetilde{\mathcal{D}}_{B})\subseteq\Phi(\mathcal{C}_{A},\widetilde{\mathcal{C}}_{B})=
\mathcal{W}_{T}\cap\Phi(\mathcal{A}_{A},\widetilde{\mathcal{A}}_{B})$, $j^{\ast}(\Phi(\mathcal{C}_{A},\widetilde{\mathcal{C}}_{B}))\subseteq  \widetilde{\mathcal{C}}_{B}=\widetilde{\mathcal{W}}_{B}\cap\widetilde{\mathcal{A}}_{B}$. Thus $j_{!}h,~q$ are weak equivalences in $\mathcal{M}_{B}$. Therefore, we have isomorphisms $\mathrm{R~j_{\ast}}\circ\mathrm{L~j_{!}}(X)\cong
\mathrm{R~j_{\ast}}\circ\mathrm{L~j_{!}}(X')\cong j^{\ast}N \stackrel{q^{-1}}\cong j^{\ast}j_{!}F\stackrel{j^{\ast}j_{!}h}\cong j^{\ast}j_{!}X'\cong X' \cong X$ in $\mathrm{Ho}(\mathcal{M}_{B})$. It is easy to see that these isomorphisms are natural.

(3) This is clear.
\end{proof}

Now, we are ready to give our main result of this section.

\begin{theorem}\label{the4.6}Let $T=\left(\begin{matrix}  A & M \\  0 & B \\\end{matrix}\right)$ be an upper triangular matrix ring, and let $\mathcal{M}_{A}=(\mathcal{A}_{A},\mathcal{W}_{A},\mathcal{D}_{A})$ and $\mathcal{M}_{B}=(\widetilde{\mathcal{A}}_{B},\widetilde{\mathcal{W}}_{B},\widetilde{\mathcal{D}}_{B})$ be the cofibrantly generated abelian model structures on $A$-$\mathrm{Mod}$ and $B$-$\mathrm{Mod}$, respectively. If $\mathrm{Tor}_{1}^{B}(M,X)=0$ for any $X\in \widetilde{\mathcal{A}}_{B}$ and $M$ is perfect relative to $\mathcal{M}_{A}$ and $\mathcal{M}_{B}$, then $\mathcal{M}_{T}=(\Phi(\mathcal{A}_{A},\widetilde{\mathcal{A}}_{B}),\mathcal{W}_{T},\mathrm{Rep}(\mathcal{D}_{A},\widetilde{\mathcal{D}}_{B}))$ is a cofibrantly generated abelian model structure on $T$-Mod and we have a recollement as shown below
  $$\xymatrixcolsep{4pc}\xymatrix{\mathrm{Ho}(\mathcal{M}_{A})
  \ar[r]^{\mathrm{L~i_{\ast}~}\cong\mathrm{~R~i_{\ast}}}&
  \ar@<-4ex>[l]_{\mathrm{L~i^{\ast}}}\ar@<3ex>[l]^{\mathrm{R~i^{!}}}\mathrm{Ho}(\mathcal{M}_{T})
\ar[r]^{\mathrm{L~j^{\ast}~}\cong\mathrm{~R~j^{\ast}}}&\ar@<-4ex>[l]_{\mathrm{L~j_{!}}}\ar@<3ex>[l]^{\mathrm{R~j_{\ast}}}
\mathrm{Ho}(\mathcal{M}_{B}),}$$
where $\mathrm{L~i^{\ast}}$, $\mathrm{L~i_{\ast}}$, $\mathrm{L~j_{!}}$, $\mathrm{L~j^{\ast}}$, $\mathrm{R~i_{\ast}}$,  $\mathrm{R~j^{\ast}}$, $\mathrm{R~i^{!}}$ and $\mathrm{R~j_{\ast}}$ are the total derived functors of those in Section 3.
\end{theorem}
\begin{proof} By Propositions \ref{pro4.4} and \ref{pro4.5}, we only need to show that there are natural isomorphisms $\mathrm{L~i_{\ast}}\cong \mathrm{R~i_{\ast}}$ and $\mathrm{L~j^{\ast}}\cong \mathrm{R~j^{\ast}}$. Let $\mathcal{B}_{A}:=\mathcal{W}_{A}\cap\mathcal{D}_{A}$, $\mathcal{C}_{A}:=\mathcal{A}_{A}\cap\mathcal{W}_{A}$, $\widetilde{\mathcal{B}}_{B}:=\widetilde{\mathcal{W}}_{B}\cap\widetilde{\mathcal{D}}_{B}$ and $\widetilde{\mathcal{C}}_{B}:=\widetilde{\mathcal{A}}_{B}\cap\widetilde{\mathcal{W}}_{B}$. Let $f: X\rightarrow Y$ be a morphism in $\mathrm{Ho}(\mathcal{M}_{T})$. The functor $\mathrm{L~i_{\ast}}$ acts by $f\mapsto \overline{f}$, where $\overline{f}$ is any morphism making the diagram below commute
$$\xymatrix{
  0   \ar[r]^{ } & i_{\ast}K_{1}  \ar[r]^{} & i_{\ast}P_{1}\ar[d]_{\overline{f}} \ar[r]^{j_{1}} & i_{\ast}X \ar[d]_{i_{\ast}f} \ar[r]^{} & 0  \\
  0 \ar[r]^{} & i_{\ast}K_{2} \ar[r]^{} & i_{\ast}P_{2} \ar[r]^{j_{2}} & i_{\ast}Y \ar[r]^{} & 0.   }$$
Here the rows are exact, $P_{1},~P_{2}\in\mathcal{A}_{A}$, and $K_{1},~K_{2}\in \mathcal{D}_{A}\cap \mathcal{W}_{A}=\mathcal{B}_{A}$. The functor $\mathrm{R~i_{\ast}}$ acts by $f\mapsto \widehat{f}$, where $\widehat{f}$ is any morphism making the next diagram commute
$$\xymatrix{
  0  \ar[r]^{} & i_{\ast}X \ar[d]_{} \ar[r]^{q_{1}} & i_{\ast}D_{1} \ar[d]_{\widehat{f}} \ar[r]^{} & i_{\ast}C_{1} \ar[r]^{} & 0 \\
  0 \ar[r]^{} & i_{\ast}Y \ar[r]^{q_{2}} & i_{\ast}D_{2} \ar[r]^{} & i_{\ast}C_{2} \ar[r]^{} & 0,   }$$
where $D_{1},~D_{2}\in\mathcal{D}_{A}$, $C_{1}~C_{2}\in\mathcal{A}_{A}\cap \mathcal{W}_{A}=\mathcal{C}_{A}$. Note that $i_{\ast}(\mathcal{B}_{A})\subseteq \mathrm{Rep}(\mathcal{B}_{A},\widetilde{\mathcal{B}}_{B})=\mathcal{W}_{T}\cap\mathrm{Rep}(\mathcal{D}_{A},\widetilde{\mathcal{D}}_{B})$ and $i_{\ast}(\mathcal{C}_{A})\subseteq \Phi(\mathcal{C}_{A},\widetilde{\mathcal{C}}_{B})=\mathcal{W}_{T}\cap\Phi(\mathcal{A}_{A},\widetilde{\mathcal{A}}_{B})$, then $j_{1},~j_{2},~q_{1}$ and $q_{2}$ are weak equivalences in $\mathcal{M}_{T}$. Hence in $\mathrm{Ho}(\mathcal{M}_{T})$, we have a commutative diagram
$$\xymatrix{
   i_{\ast}P_{1} \ar[d]_{\overline{f}} \ar[r]^{j_{1}} &  i_{\ast}X \ar[d]_{} \ar[r]^{q_{1}} &  i_{\ast}D_{1} \ar[d]^{\widehat{f}} \\
   i_{\ast}P_{2} \ar[r]^{j_{2}} &  i_{\ast}Y \ar[r]^{q_{2}} &  i_{\ast}D_{2} }
$$
giving rise to a natural isomorphism $\mathrm{L~i_{\ast}}\cong \mathrm{R~i_{\ast}}$.
The proof of the natural isomorphism $\mathrm{L~j^{\ast}}\cong \mathrm{R~j^{\ast}}$ is similar.
\end{proof}

\begin{remark}\label{rem4.7}
 Let $\mathcal{M}_{A}=(\mathcal{A}_{A},\mathcal{W}_{A},\mathcal{D}_{A})$ and $\mathcal{M}_{B}=(\widetilde{\mathcal{A}}_{B},\widetilde{\mathcal{W}}_{B},\widetilde{\mathcal{D}}_{B})$ be the cofibrantly generated abelian model structures on $A$-$\mathrm{Mod}$ and $B$-$\mathrm{Mod}$, respectively, and let $\mathcal{B}_{A}:=\mathcal{W}_{A}\cap\mathcal{D}_{A}$, $\mathcal{C}_{A}:=\mathcal{A}_{A}\cap\mathcal{W}_{A}$, $\widetilde{\mathcal{B}}_{B}:=\widetilde{\mathcal{W}}_{B}\cap\widetilde{\mathcal{D}}_{B}$ and $\widetilde{\mathcal{C}}_{B}:=\widetilde{\mathcal{A}}_{B}\cap\widetilde{\mathcal{W}}_{B}$.
By Theorem \ref{the3.8}, if  $\mathrm{Ext}_{A}^{1}(M,F)=0$ for any $F\in \mathcal{D}_{A}$, we have two complete hereditary cotorsion pairs $(\mathrm{Rep}(\mathcal{A}_{A},\widetilde{\mathcal{A}}_{B}), \Psi(\mathcal{B}_{A},\widetilde{\mathcal{B}}_{B}))$ and $(\mathrm{Rep}(\mathcal{C}_{A},\widetilde{\mathcal{C}}_{A}), \Psi(\mathcal{D}_{A},\widetilde{\mathcal{D}}_{B}))$.
Similarly, we can see that under some hypotheses, there exist a cofibrantly generated model structure $\mathcal{M}'_{T}=(\mathrm{Rep}(\mathcal{A}_{A},\widetilde{\mathcal{A}}_{B}),\mathcal{W}'_{T},\Psi(\mathcal{D}_{A},\widetilde{\mathcal{D}}_{B}))$  on $T$-Mod and a recollement of $\mathrm{Ho}(\mathcal{M}'_{T})$ relative to $\mathrm{Ho}(\mathcal{M}_{A})$ and $\mathrm{Ho}(\mathcal{M}_{B}).$
\end{remark}

Finally, we give some applications of Theorem \ref{the4.6} to Gorenstein homological algebra.

Let $R$ be a ring.
Recall that a left $R$-module $M$ is \emph{Gorenstein projective} \cite{ee00} if $M=\mathrm{ker}d_{\mathbb{P}}^{0}$ for some exact complex of projective left $R$-modules
$$\mathbb{P}:\cdots\rightarrow
P^{-1}\stackrel{d_{\mathbb{P}}^{-1}}\rightarrow P^{0} \stackrel{d_{\mathbb{P}}^{0}}\rightarrow P^{1}\stackrel{}\rightarrow \cdots$$
 which remains exact after applying $\mathrm{Hom}_{R}(-,Q)$ for any projective left $R$-module $Q$. Denote by $\mathcal{GP}_{R}$ the class of Gorenstein projective left $R$-modules. Recall that a ring $R$ is \emph{left Gorenstein regular} \cite[Definition 2.1]{ee14} if and only if each projective left $R$-module has finite injective dimension and each injective left $R$-module has finite projective dimension. Each Gorenstein ring is left Gorenstein regular (see \cite[Theorem 9.1.11]{ee00}). By \cite[Theorem 3.5]{ee14}, if $_{A}M$ has finite projective dimension, $M_{B}$  has finite flat
 dimension and that $A$ is left Gorenstein regular, then $\mathcal{GP}_{T}=\Phi(\mathcal{GP})$.

Let $R$ be a ring with all projective left $R$-modules having finite injective dimension. Then $(\mathcal{GP}_{R},{\mathcal{GP}_{R}^{\perp}})$ forms a hereditary cotorsion pair which generated by a set by \cite[Theorem 4.2]{wa16}. Moreover, we have $\mathcal{GP}_{R}\cap{\mathcal{GP}_{R}^{\perp}}=\mathcal{P}_{R}$ by \cite{gj16}. Therefore, by Lemma \ref{lem4.2}, there exists a category $\mathcal{W}_{R}$ of modules, such that $(\mathcal{GP}_{R},\mathcal{W}_{R},R\text{-Mod})$ forms a hereditary abelian model structure.

Recall that for any ring $R$, the big singularity category $D_{Sg}(R)=D^{b}(R\text{-Mod})/K^{b}(\mathrm{Proj}\text{-}R)$ (see \cite{ch11} and \cite[Section 6]{be00}). If $R$ is Gorenstein regular, then \cite[Theorem 4.1]{ie12} and \cite[Theorem 4.16]{be00} imply that we have a triangle equivalence $\underline{\mathcal{GP}_{R}}\overset{\sim}{\rightarrow} D_{Sg}(R)$ (see also \cite[Corollary 3.4]{we16}). Thus, as a consequence of Theorem \ref{the4.6}, we obtain the next recollements. One can compare it with \cite[Theorem 3.5]{zh13}.

\begin{corollary}\label{cor5.1}
Let $T=\left(\begin{matrix}  A & M \\  0 & B \\\end{matrix}\right)$ be an upper triangular matrix ring. Suppose $_{A}M$ has finite projective dimension, $M_{B}$  has finite flat dimension, and $T$ is left Gorenstein regular. Then we have a recollement
$$\xymatrix{\underline{\mathcal{GP}_{A}}\ar[r]^{}&\ar@<-2ex>[l]!R|{}
\ar@<2ex>[l]!R|{}\underline{\mathcal{GP}_{T}}
\ar[r]!L|{}&\ar@<-2ex>[l]!L|{}\ar@<2ex>[l]!L|{}\underline{\mathcal{GP}_{B}},}$$
or equivalently, a recollement
$$\xymatrix{D_{Sg}(A)\ar[r]^{}&\ar@<-2ex>[l]!R|{}\ar@<2ex>[l]!R|{}D_{Sg}(T)
\ar[r]!L|{}&\ar@<-2ex>[l]!L|{}\ar@<2ex>[l]!L|{}D_{Sg}(B).}$$
\end{corollary}
\begin{proof}
By assumption and \cite[Theorem 3.1]{ee14}, $T$ is left Gorenstein regular if and only if $A$ and $B$ are left Gorenstein regular. Then $A$ and $B$ are rings with all projective left modules having finite injective dimension.
Let $\mathcal{M}_{A}=(\mathcal{GP}_{A},\mathcal{W}_{R},A\text{-Mod})$ and $\mathcal{M}_{A}=(\mathcal{GP}_{B},\mathcal{W}_{B},B\text{-Mod})$ be abelian model structures on $A\text{-Mod}$ and $B\text{-Mod}$, respectively. Since $M_{B}$  has finite flat dimension, by the similar argument in \cite[Lemma 4.1]{xi08}, we have $\mathrm{Tor}_{1}^{B}(M,E)=0$ for any $E\in \mathcal{GP}_{B}$. Furthermore, it is easy to see that the bimodule $M$ is perfect relative to $\mathcal{M}_{A}$ and $\mathcal{M}_{B}$. Thus the recollements follow from Theorem \ref{the4.6}.
\end{proof}

Recall that a left $R$-module $M$ is \emph{Gorenstein injective} \cite{ee00} if $M=\mathrm{ker}d_{\mathbb{I}}^{0}$ for some exact complex of injective left $R$-modules $$\mathbb{I}:\cdots\rightarrow
I^{-1}\stackrel{d_{\mathbb{I}}^{-1}}\rightarrow I^{0} \stackrel{d_{\mathbb{I}}^{0}}\rightarrow I^{1}\stackrel{}\rightarrow \cdots$$
 which remains exact after applying $\mathrm{Hom}_{R}(E,-)$ for any injective left $R$-module $E$. Denote by $\mathcal{GI}_{R}$ the class of Gorenstein injective left $R$-modules. By \cite[Theorem 3.8]{ee14}, if $_{A}M$ has finite projective dimension, $M_{B}$  has finite flat
 dimension and $B$ is left Gorenstein regular, then $\mathcal{GI}_{T}=\Psi(\mathcal{GI})$.
  Recall that a triplet $\mathcal{(F,H,L)}$ is called a small hereditary \emph{cotorsion triple} \cite{ch10} if $\mathcal{(F,H)}$ and  $\mathcal{(H,L)}$ are small hereditary cotorsion pairs. We have the following equivalence.
\begin{proposition}Let $T=\left(\begin{matrix}  A & M \\  0 & B \\\end{matrix}\right)$ be an upper triangular matrix ring. If  $T$ is a Gorenstein ring with $_{A}M$ having finite projective dimension, then there is an equivalence of triangulated categories $$\Phi(\mathcal{GP})/\sim\overset{\sim}{\rightarrow} \Psi(\mathcal{GI})/\sim.$$
\end{proposition}
\begin{proof} By \cite[Lemma 3.1]{wr16}, if $T$ is a Gorenstein ring with $_{A}M$ having finite projective dimension, then $A$ and $B$ are Gorenstein rings. Moreover, $M_{B}$  has finite flat
 dimension. Let  $\mathcal{W}$ be the category of modules with finite projective (injective) dimension.  In this case, we can apply \cite[Theorems 8.3 and 8.4]{ho02} and obtain small hereditary cotorsion triples $(\mathcal{GP}_{A},\mathcal{W}_{A},\mathcal{GI}_{A})$ and $(\mathcal{GP}_{B},\mathcal{W}_{B},\mathcal{GI}_{B})$ in $A$-Mod and $B$-Mod respectively. By \cite[Theorem 2.22]{ho04}, $\mathrm{Ext}_{A}^{1}(M,F)=0$ for any $F\in \mathcal{GI}_{A}$. Therefore, $\mathcal{M}_{1}=(\Phi(\mathcal{GP}),\mathrm{Rep}(\mathcal{W}),~T\text{-}\mathrm{Mod})$ and  $\mathcal{M}_{2}=(T\text{-}\mathrm{Mod},~\mathrm{Rep}(\mathcal{W}),\Psi(\mathcal{GI}))$ are hereditary Hovey triples by Corollary \ref{cor3.9}. Moreover, $(\mathrm{id,id}):\mathcal{M}_{1}\rightleftarrows\mathcal{M}_{2} $ is  a Quillen equivalence by \cite[Lemma 3.2]{re18}. Hence, according to \cite[Proposition 16.2.3]{ma12}, we have an adjoint triangulated equivalence $(\mathrm{L~id},\mathrm{R~id}): \mathrm{Ho}(\mathcal{M}_{1})\rightleftarrows\mathrm{Ho}(\mathcal{M}_{2})$, or equivalently, $\Phi(\mathcal{GP})/\sim\overset{\sim}{\rightarrow} \Psi(\mathcal{GI})/\sim$.
\end{proof}

\bigskip
Recall that a left $R$-module $M$ is called\emph{ Gorenstein flat} \cite{ee00} if $M=\mathrm{ker}d_{\mathbb{F}}^{0}$
for some exact complex of flat left $R$-modules
$$\mathbb{F}:\cdots\rightarrow
F^{-1}\stackrel{d_{\mathbb{F}}^{-1}}\rightarrow F^{0} \stackrel{d_{\mathbb{F}}^{0}}\rightarrow F^{1}\stackrel{}\rightarrow \cdots$$
 which remains exact after applying $I\otimes_{R}-$ for any injective right $R$-module $I$. Denote by $\mathcal{GF}_{R}$ the class of Gorenstein flat left $R$-modules.
Next, we give the characterization of Gorenstein flat left $T$-modules.

\begin{lemma}\label{lem5.2} Suppose each injective left $T$-module has finite projective dimension, $_{A}M$ has finite projective dimension, $M_{B}$  has finite flat dimension. Let $X=\binom{X_{1}}{X_{2}}_{\phi^{X}}$ be a left $T$-module. Then
$X\in \mathcal{GF}_{T}~\text{if and only if}~X\in\Phi(\mathcal{GF}).$
\end{lemma}
\begin{proof}
Since each injective left $T$-module has finite projective dimension, any injective left $A$-module and injective left $B$-module have finite projective dimension by \cite[Theorem 3.1]{ee14}. Hence, if $X\in \mathcal{GF}_{T}$, by the argument similar to that in \cite[Proposition 3.5]{zh16}, we get $X\in\Phi(\mathcal{GF})$.

Conversely, if $X\in\Phi(\mathcal{GF})$, then there exists a short exact sequence of left $T$-modules
$$0\rightarrow\left(\begin{matrix}  M\otimes_{B}X_{2}  \\ X_{2}  \\\end{matrix}\right)_{id}\rightarrow\left(\begin{matrix}  X_{1}  \\ X_{2} \\\end{matrix}\right)_{\phi^{X}}\rightarrow\left(\begin{matrix}  \mathrm{coker}\phi^{X}  \\ 0 \\\end{matrix}\right)_{0}\rightarrow0.
$$
By \cite{js18}, the class $\mathcal{GF}$ is always closed under extensions, regardless of the ring. So we only have to verify that $\left(\begin{smallmatrix}  M\otimes_{B}X_{2}  \\ X_{2}  \\\end{smallmatrix}\right)$ and $\left(\begin{smallmatrix}  \mathrm{coker}\phi^{X}  \\ 0 \\\end{smallmatrix}\right)$ are Gorenstein flat.
Since $X_{2}$ is Gorenstein flat, there is an exact complex $F$  consisting of flat left $B$-modules, which remains exact after applying $I\otimes_{R}-$ for any injective right $B$-module $I$ and such that $\mathrm{ker}d_{F}^{0}=X_{2}$. Using \cite[Lemma 2.3]{ee14}, we get that the complex $M\otimes_{B}F$ is exact in $B$-Mod, which implies that $j_{!}F$ is exact in $T$-Mod. Since each injective left $T$-module has finite projective dimension, by \cite[Lemma 2.3]{ee14} again, $E\otimes_{T}j_{!}F$ is exact for any injective right $T$-module $E$. Therefore,  $\mathrm{ker}d_{j_{!}F}^{0}=j_{!}X_{2}=\left(\begin{smallmatrix}  M\otimes_{B}X_{2}  \\ X_{2}  \\\end{smallmatrix}\right)$ is Gorenstein flat. Similarly, it is easy to verify that $i_{\ast}(\mathrm{coker}\phi^{X})=\left(\begin{smallmatrix}  \mathrm{coker}\phi^{X}  \\ 0 \\\end{smallmatrix}\right)$ is Gorenstein flat.
\end{proof}

\begin{lemma}\label{lem5.3} Let $M$ be a right $R$-module with finite flat dimension and $G$ be a Gorenstein flat left $R$-module. Then $\mathrm{Tor}_{i}^{R}(M,G)=0$ for all $i>0$.
\end{lemma}
\begin{proof} The idea of the proof given here is
essentially taken from \cite[Lemma 4.1]{xi08}.
Denote by $\mathrm{fd}(M)$ the flat dimension of $M$. Suppose $M$ is a right $R$-module with fd$(M)\leq1$. Then we have a flat resolution of $M$:
$0\rightarrow F_{1} \rightarrow F_{0}\rightarrow M\rightarrow 0$
with $F_{0}$ and $F_{1}$ flat. From this exact sequence, we obtain the following exact sequence:
$$0\rightarrow \mathrm{Tor}^{R}_{1}(M,G)\rightarrow F_{1}\otimes_{R}G\stackrel{\alpha}\rightarrow F_{0}\otimes_{R}G\rightarrow M\otimes_{R}G\rightarrow0.$$
Since $G$ is Gorenstein flat, we have an exact sequence $0\rightarrow G\rightarrow L^{0}$ with $L^{0}$ flat. To see $\mathrm{Tor}^{R}_{1}(M,G)=0$, we shall show that $\alpha$ is a monomorphism. This follows from the following exact commutative diagram:
$$\xymatrix{
      & 0 \ar[d]_{} & 0 \ar[d]_{}  &  &  \\
   & F_{1}\otimes_{R}G \ar[d]_{} \ar[r]^{\alpha} & F_{0}\otimes_{R}G \ar[d]_{} \ar[r]^{} & M\otimes_{R}G \ar[d]_{} \ar[r]^{} & 0  \\
  0 \ar[r]^{} & F_{1}\otimes_{R}L^{0} \ar[r]^{} & F_{0}\otimes_{R}L^{0} \ar[r]^{} & M\otimes_{R}L^{0} \ar[r]^{} & 0.   }$$
Hence $\mathrm{Tor}^{R}_{i}(M,G)=0$ for all $i > 0$ and every Gorenstein flat left $R$-module $G$ if fd$(M)\leq1$.
Now suppose fd$(M)=n>1$. Then we have a flat resolution of $M$:
$$0 \longrightarrow F_{n} \stackrel{d_{n}}\longrightarrow F_{n-1}\stackrel{d_{n-1}}\longrightarrow\cdots\longrightarrow F_{1} \stackrel{d_{1}} \longrightarrow F_{0}\longrightarrow M\longrightarrow 0.$$
Let $C$ be the image of $d_{n-1}$. Then fd$(C)\leq1$. Since $G$ is Gorenstein flat, there exists an exact sequence
$$0 \longrightarrow G \stackrel{}\longrightarrow L^{0}\stackrel{\delta^{0}}\longrightarrow L^{1}\stackrel{\delta^{1}}\longrightarrow L^{2}\stackrel{}\longrightarrow\cdots$$
with $L^{i}$ flat for all $i\geq0$. Let $K_{i}$ be the image of $\delta^{i}$ and $K_{-1}=G$. Note that all $K_{i}$ are Gorenstein flat.
Given a positive integer $j$. If $j\geq n+1$, then $\mathrm{Tor}^{R}_{j}(M,G)=0$. Suppose $j\leq n$. Let $i = n-j-1$. Then
$$\mathrm{Tor}^{R}_{j}(M,G)\cong\mathrm{Tor}^{R}_{j+i+1}(M,K_{i})\cong\mathrm{Tor}^{R}_{1}(C,K_{i})=0.$$
Thus $\mathrm{Tor}^{R}_{j}(M,G)=0$ for all $j >0$, as desired.
\end{proof}

It is well known that $(\mathcal{F}_{R},\mathcal{C}_{R})$ is a complete hereditary cotorsion pair for any ring $R$, where $\mathcal{F}_{R}$ is the class of flat left $R$-modules, $\mathcal{C}_{R}=\mathcal{F}_{R}^{\perp}$ is the class of cotorsion left $R$-modules. On the other hand, \v{S}aroch and \v{S}\v{t}ov\'{\i}\v{c}ek \cite{js18} have recently proved that the pair $(\mathcal{GF}_{R},\mathcal{GF}_{R}^{\perp})$ is a perfect (so, in particular, complete) and hereditary cotorsion pair for any ring. According to \cite[Proposition 4.1]{es17}, we see that $\mathcal{GF}_{R}\cap\mathcal{GF}_{R}^{\perp}=\mathcal{F}_{R}\cap\mathcal{C}_{R}$. Thus, the cotorsion pairs $(\mathcal{GF}_{R},\mathcal{GF}_{R}^{\perp})$ and $(\mathcal{F}_{R},\mathcal{C}_{R})$ satisfy the conditions (1) and (2) in Lemma \ref{lem4.2}. Then there exists a category $\mathcal{W}'_{R}$ of modules, such that $(\mathcal{GF}_{R},\mathcal{W}'_{R}, \mathcal{C}_{R})$ forms a hereditary abelian model structure.
As a consequence of the above lemmas, we obtain the following recollements.

\begin{theorem}\label{cor5.3} Let $T=\left(\begin{matrix}  A & M \\  0 & B \\\end{matrix}\right)$ be an upper triangular matrix ring. Suppose $_{A}M$ has finite projective dimension, $M_{B}$  has finite flat dimension, and each injective left $T$-module has finite projective dimension.  Then we have a recollement
$$\xymatrix{\underline{\mathcal{GF}_{A}\cap \mathcal{C}_{A}}\ar[r]^{}&\ar@<-2ex>[l]!R|{}\ar@<2ex>[l]!R|{}\underline{\mathcal{GF}_{T}\cap \mathcal{C}_{T}}
\ar[r]!L|{}&\ar@<-2ex>[l]!L|{}\ar@<2ex>[l]!L|{}\underline{\mathcal{GF}_{B}\cap \mathcal{C}_{B}}}.$$

 Moreover, if $T$ is Gorenstein, we have the recollement
  $$\xymatrix{D^{b}(A)/K^{b}(\mathcal{F}_{A}\cap \mathcal{C}_{A})\ar[r]^{}&\ar@<-2ex>[l]!R|{}\ar@<2ex>[l]!R|{}D^{b}(T)/K^{b}(\mathcal{F}_{T}\cap \mathcal{C}_{T})
\ar[r]!L|{}&\ar@<-2ex>[l]!L|{}\ar@<2ex>[l]!L|{}D^{b}(B)/K^{b}(\mathcal{F}_{B}\cap \mathcal{C}_{B}).}$$
\end{theorem}
\begin{proof}Let $\mathcal{M}_{A}=(\mathcal{GF}_{A},\mathcal{W}'_{A},\mathcal{C}_{A})$ and $\mathcal{M}_{B}=(\mathcal{GF}_{B},\mathcal{W}'_{B},\mathcal{C}_{B})$ be abelian model structures on $A\text{-Mod}$ and $B\text{-Mod}$, respectively. It is easy to see that the bimodule $M$ is perfect relative to $\mathcal{M}_{A}$ and $\mathcal{M}_{B}$.
By Lemma \ref{lem5.2} and \cite[Theorem 2.8]{ee11}, we see that $\Phi(\mathcal{GF})=\mathcal{GF}_{T}$ and $\mathrm{Rep}(\mathcal{C})=\mathcal{C}_{T}$. As a consequence of Lemma \ref{lem5.3} and Theorem \ref{the4.6}, we get the first recollement. If $T$ is a Gorenstein ring, \cite[Lemma 3.1]{wr16} tells us that $A$ and $B$ are also Gorenstein rings. Therefore, the second recollement follows from the triangle equivalence
$$\underline{\mathcal{GF}_{R}\cap \mathcal{C}_{R}}\overset{\sim}{\rightarrow} D^{b}(R)/K^{b}(\mathcal{F}_{R}\cap \mathcal{C}_{R})$$
for any Gorenstein ring $R$ by \cite[Corollary 5.2]{di18}.
\end{proof}

\bigskip \centerline {\bf Acknowledgements}
\bigskip

\hspace{-0.5cm}  This research was partially supported by the National Natural Science Foundation of China (Nos. 11771202, 11361052). Part of this work was carried out while the first author was visiting  Charles University in Prague. He gratefully acknowledges the financial support from China Scholarship Council (CSC No. 201806190107) and the kind hospitality from the host university.

\end{document}